\documentclass[12pt,draft,leqno]{article}
\usepackage{amssymb,eucal,latexsym}
\usepackage{amsmath,amsthm,latexsym,amscd,amsbsy,fleqn,graphicx}
\usepackage{xy}
\xyoption{all}

\newtheorem{theorem}{Theorem}[section]
\newtheorem{lm}[theorem]{Lemma}

\newtheorem{exas}[theorem]{Examples}
\newtheorem{cor}[theorem]{Corollary}
\newtheorem{pro}[theorem]{Proposition}
\newtheorem{defi}[theorem]{Definition}

\newtheorem{rem}[theorem]{Remark}

\newtheorem{nist}[theorem]{}
\newtheorem{const}[theorem]{Construction}

\newcommand{\timplies}{\text{ implies }}

\newcommand{\tiff}{if, and only if, }

\def\p{\varphi}
\def\a{\alpha}
\def\b{\beta}
\def\d{\delta}
\def\ep{\varepsilon}
\def\g{\gamma}
\def\GA{\Gamma}

\def\vk{\varkappa}
\def\la{\lambda}

\def\t{\theta}

\def\s{\sigma}

\def\om{\omega}
\def\z{\zeta}

\def\ra{\rightarrow}
\def\Lra{\Longrightarrow}

\def\lra{\longrightarrow}

\def\sbe{\subseteq}

\def\stm{\setminus}
\def\ems{\emptyset}

\def\ti{\tilde}

\def\ex{\exists}

\def\we{\wedge}
\def\bw{\bigwedge}
\def\bv{\bigvee}

\def\ap{^{\,\prime}}
\def\inv{^{-1}}
\def\st{\ |\ }

\def\nin{\not\in}

\def\1{{\bf 1}}
\def\2{\mbox{{\sf 2}}}
\def\3{\mbox{{\bf 3}}}

\def\AA{{\cal A}}
\def\BB{{\cal B}}
\def\CC{{\cal C}}

\def\EE{{\cal E}}

\def\JJ{{\cal J}}

\def\PP{{\cal P}}

\def\XX{{\cal X}}
\def\YY{{\cal Y}}

\def\PPP{{\sf P}}

\newcommand{\RC}{\operatorname{RC}}

\def\rc{{\sf RC}}
\def\ro{{\sf RO}}

\def\smf{{\,\smallfrown\,}}
\def\smfc{{\,\smallfrown_\circ\,}}

\def\at{{\sf At}}

\def\DHC{{\bf deV}}

\def\HC{{\bf KHaus}}

\def\dom{{\rm dom}}

\def\cod{{\rm cod}}

\def\int{\mbox{{\rm int}}}

\def\cl{\mbox{{\rm cl}}}

\def\Clustsf{\mbox{{\sf Clust}}}

\def\Ultsf{\mbox{{\sf Ult}}}
\def\RCsf{\mbox{{\sf RC}}}
\def\COsf{\mbox{{\sf CO}}}

\def\Tych{{\bf Tych}}

\def\Set{{\bf Set}}

\def\tcx{t_X^C}
\def\tcy{t_Y^C}
\def\tcx0{t_{(X,X_0)}}
\def\tcy0{t_{(Y,Y_0)}}

\def\di{\diamond}

\def\bU0{\bar{U}=(U^0,(U^i,U^{ci})_{i\in\omega})}

\def\bV0{\bar{V}=(V^0,(V^i,V^{ci})_{i\in\omega})}

\def\sfr{\smallfrown}

\title{{\LARGE\bf Categorical Extension of Dualities: }\\
\vspace{0.2cm}
{\LARGE\bf From Stone to de Vries and Beyond, I}\\
\vspace{0.5cm}
{\large\bf G. Dimov, E. Ivanova-Dimova
and W. Tholen}\thanks{The  first  author acknowledges the   support
by Bulgarian National Science Fund, contract no. DN02/15/19.12.2016.
The second author acknowledges the   support by the Bulgarian Ministry of
Education and Science under the National Research Programme “Young scientists and postdoctoral students” approved by DCM \# 577/17.08.2018.
The third author acknowledges the support under Discovery Grant no. 501260 of the Natural Sciences and Engineering Council of Canada.}
\\
\vspace{0.2cm}
{\footnotesize\rm Faculty of Mathematics and Inf.,  Sofia University,}
{\footnotesize\rm 5 J. Bourchier Blvd., 1164 Sofia, Bulgaria}\\
{\footnotesize\rm Dept. of Mathematics and Statistics, York University,}
 {\footnotesize\rm Toronto, Ontario, M3J 1P3, Canada}
}

\author{}

\date{}

\begin{document}

\maketitle

\begin{abstract}
Propounding
a general categorical framework for the extension of dualities,
we present a new proof of the de Vries Duality Theorem
for the category $\bf KHaus$ of compact Hausdorff spaces and their continuous maps, as an extension of a restricted Stone duality. Then, applying a dualization of the categorical framework to the de Vries duality, we give an alternative proof of the extension of the de Vries duality to the category $\bf Tych$ of Tychonoff spaces that was provided by Bezhanishvili, Morandi and Olberding.
In the process of doing so, we obtain new duality theorems for both categories, $\bf{KHaus}$ and $\Tych$.
\end{abstract}

\footnotetext[1]{{\footnotesize
{\em Keywords:}  compact Hausdorff space, Tychonoff space, Stone space, regular closed/open set, irreducible map, quasi-open map, projective cover, (complete) Boolean algebra, (normal) contact algebra, de Vries algebra, ultrafilter, cluster, right/left lifting of a dual adjunction, semi-right adjoint functor, covering class,  Stone duality, Tarski duality, de Vries duality, (universal) de Vries pair, Booleanization of a de Vries pair, (universal) Boolean de Vries extension.}}

\footnotetext[2]{{\footnotesize
{\em 2010 Mathematics Subject Classification:}    54D30, 54D15, 18A40, 18B30, 54E05, 54C10, 54G05, 06E15, 54H10, 03G05.}}

\footnotetext[3]{{\footnotesize {\em E-mail addresses:}
gdimov@fmi.uni-sofia.bg, elza@fmi.uni-sofia.bg, tholen@mathstat.yorku.ca}.}

\section{Introduction}
A restriction of the Stone duality renders the category of extremally disconnected compact Hausdorff spaces and their continuous maps as dually equivalent to the category of complete Boolean algebras and Boolean homomorphisms. Extending this duality, in 1982 de Vries \cite{deV} showed that the category $\bf KHaus$ of all compact Hausdorff spaces is dually equivalent to the category $\bf deV$, whose objects are complete Boolean algebras that come equipped with a so-called {\em normal contact relation}, originally called compingent algebras, but now known as {\em de Vries algebras}. The morphisms of $\bf deV$ are maps satisfying certain compatibility conditions with the Boolean structure and the contact relation, but they are not necessarily Boolean homomorphisms; also, their categorical composition is generally not given by the ordinary map composition.

The de Vries duality
is realized by the contravariant functors $\sf RC$ and $\sf Clust$.
 The functor $\sf RC$ assigns to every space $X$ in $\bf KHaus$ the complete Boolean algebra ${\sf RC}(X)$ of its regular closed sets, provided with the relation that declares two sets to be in contact when they intersect.
 (Recall that
 for $X$ extremally disconnected, ${\sf RC}(X)$ coincides with the algebra ${\sf CO}(X)$ of closed and open sets in $X$, as used in the Stone duality.)
 The functor $\sf Clust$ generalizes Stone's formation of the space of ultrafilters in a (complete) Boolean algebra, assigning to a de Vries algebra its space of so-called {\em clusters}. Of course, ${\sf RC}(X)$ may equivalently be replaced by the de Vries algebra ${\sf RO}(X)$ of regular open sets in $X$ and, as we show in this paper, the functor $\sf Clust$ is (just like  the ultrafilter functor for Boolean algebras) represented by the two-chain $\sf 2$ and may be replaced by the $\bf KHaus$-valued contravariant hom-functor ${\bf deV}(-,{\sf 2})$.

\begin{center}
$\xymatrix{{\bf deV}^{\rm op}\ar@/^0.4pc/[rr]^{\sf Clust\;\;}\ar@{}[rd]|{\rm{}} & {\scriptstyle \simeq}&{\bf KHaus}\ar@/^0.4pc/[ll]^{\sf RC}\\
{\bf CBoo}^{\rm op}\ar@{^(->}[u]^{}\ar@/^0.4pc/[rr]^{\sf Ult} & {\scriptstyle \simeq}&{\bf EKH}\ar@/^0.4pc/[ll]^{\sf CO}\ar@{^(->}[u]_{}}$
\hfil
$\xymatrix{{\bf UBdeV}^{\rm op}\ar@/^0.4pc/[rr]^{\sf \;\;}\ar@{}[rd]|{\rm{}} & {\scriptstyle \simeq}&{\bf Tych}\ar@/^0.4pc/[ll]^{\sf }\\
{\bf deV}^{\rm op}\ar@{^(->}[u]^{}\ar@/^0.4pc/[rr]^{{\bf deV}(-,{\sf 2})} & {\scriptstyle \simeq}&{\bf KHaus}\ar@/^0.4pc/[ll]^{\sf RO}\ar@{^(->}[u]_{}}$
\end{center}

In their 2018 paper \cite{BMO}, Bezhnanishvili, Morandi and Olberding (BMO) showed that the de Vries duality may be extended from ${\bf KHaus}$ to the category ${\bf Tych}$ of all Tychonoff spaces when, on the algebraic side, one considers so-called maximal de Vries extensions. These are de Vries algebras that come equipped with an embedding into a complete atomic Boolean algebra satisfying a maximality condition. As this condition corresponds precisely to the universal property of the Stone-\v{C}ech compactification of a Tychonoff space, we call them {\em universal Boolean de Vries extensions}. The resulting category ${\bf UBdeV}$ that is dually equivalent to $\bf Tych$ contains (a copy of) $\bf deV$ as a full subcategory and is itself a full subcategory of the arrow category of ${\bf deV}$.

The primary goal of this paper is to show that both duality extensions (as displayed in the above diagram) may be obtained with the help of
a general categorical
technique (which we develop here in Theorem \ref{more duality})
for the extension of a given (bottom) dual equivalence when one of the two categories is fully embedded into a larger (top) category. This approach (which in \cite{DDT} we applied to obtain an alternative proof of the Fedorchuk duality \cite{F}, using a special case of Theorem \ref{more duality})
then gives us a first equivalent description of the other (top) category in terms of the given data and,
in fact, leads us to some alternative
new dual equivalences for both categories, $\bf KHaus$ and $\bf Tych$. Furthermore, the categorical viewpoint reveals that the extension technique used for obtaining the BMO duality (on the right) is essentially dual to that one used to obtain the de Vries duality (on the left). We must say ``essentially dual'' here since, roughly speaking, the role of the projective covers used to establish the de Vries duality is taken by the Stone-\v{C}ech compactification in the case of the BMO duality, but one knows that the
 compactification is functorial while the formation of the
 cover is not \cite{AT}.
Nevertheless, the de Vries duality is obtained here with the help of Theorem \ref{more duality}, while the BMO duality is proved using the dualization of the special case of  Theorem \ref{more duality} presented in Corollary \ref{duality corefl}\footnote{Using Corollary \ref{duality corefl},  one can obtain as well the recent extension of the Stone Duality Theorem to the category of zero-dimensional Hausdorff spaces and continuous maps, as established in \cite{DD}.}.

The pursuit of the first goal naturally leads us to reaching our second goal, namely to describe a category, ${\bf deVBoo}/\!\sim$, which is equivalent to $\bf deV$ and, in fact, has the same objects as $\bf deV$, but whose morphisms are easier to handle than those of $\bf deV$. In fact, the definition of de Vries morphism does not flow naturally out of the object definition, but rather appears to be a peculiar mix of preservation conditions for some of the object structure, which justifies itself only by the fact that, in the end, it``works''. By contrast, the morphisms in ${\bf deVBoo}$ are simply Boolean homomorphisms reflecting the contact relation, to be composed by ordinary map composition, and also the equivalence relation  defining the quotient category ${\bf deVBoo}/\!\sim$ has an easy description, both categorically and in terms of the Boolean structure.\footnote{We note that in \cite{DI}, as the restriction of a duality involving the category of all locally compact Hausdorff spaces, another category dually equivalent to $\bf KHaus$
 is presented. While its composition law may be considered to be more natural than that of the category $\bf deV$, its morphisms, which are special multi-valued maps, may not.}
 In our proof of the de Vries duality, ${\bf deVBoo}/\!\sim$ acts as a ``mediator'' between $\bf KHaus$ and $\bf deV$.

Similarly, also in presenting an alternative proof of the BMO duality, we form a ``mediating'' category, $\bf UdeV$, whose definition flows naturally from our categorical extension technique. As a consequence, this category of {\em universal de Vries pairs} is easily seen to be dually equivalent to $\bf Tych$, and we then show that it is equivalent to the BMO category $\bf UBdeV$. In order for us to reach the BMO duality, in both of these categories that are dually equivalent to $\bf Tych$, we must stick with the de Vries notion of morphism. However, in the sequel \cite{DDT2} to this paper, we use the morphisms of ${\bf deVBoo}/\!\sim$ rather than those of $\bf deV$ to establish a category dually equivalent to $\bf Tych$ and thereby avoid the inconveniences pertaining to the notion of de Vries morphisms. In addition, \cite{DDT2} offers various other applications of our categorical extension technique, to obtain alternative proofs or modifications of the dualities established in \cite{BMO1, D-AMH1-10} for the categories of locally compact Hausdorff spaces and of normal Hausdorff spaces and their continuous maps, as well as some other new duality theorems.

Here is an outline of the contents and the organization of the paper. For the reader's convenience, we recall in Section 2 the needed standard facts and fix the notation pertaining to all, the Stone duality, the Tarski duality (of sets and complete atomic Boolean algebras), the de Vries duality, and its Fedorchuk variation. In fact, following a communication by Richard Garner to the third author, we present the Stone duality as obtainable from the $\bf Set$-valued functor of Boolean algebras represented by $\sf 2$, via the Eilenberg-Moore construction for its induced monad, which is the ultrafilter monad. We do so since, as we show in Section 6, the corresponding fact turns out to hold for de Vries algebras as well. At the end of the section we summarize some key topological facts pertaining to projective covers, also known as absolutes, including Gleason's Theorem and Alexandroff's Theorem.

The categorical extension constructions for dual equivalences as given in Theorem \ref{more duality} and Corollary \ref{duality refl} of Section 3 provide the categorical framework that we use to establish new proofs for, respectively, the de Vries duality and the BMO duality; moreover, in the process of doing so, we obtain new duality theorems for the categories ${\bf KHaus}$ and ${\bf Tych}$.
Our notion of $\XX$-covering class (Definition \ref{covering class}) is the common backbone for both constructions. But while Corollary \ref{duality refl} entails just the dualization of the construction that we used in \cite{DDT} for an alternative proof of the Fedorchuk duality, Theorem \ref{more duality} involves a quotient construction that is needed when we apply it towards the de Vries duality, in order to offset the non-functoriality of projective covers.

In Section 4 we give our alternative approach to the de Vries duality, via a two-step procedure. First, when applied to the restricted Stone duality, Theorem \ref{more duality} leads us to a {\em modified} de Vries duality theorem (Theorem \ref{modified deVries}), whereby $\bf KHaus$ is dually equivalent to the category ${\bf deVBoo}/\!\sim$, which we consider to be
a viable contender to $\bf deV$  for a dual representation of $\bf KHaus$. The second step for completing our alternative proof of the de Vries duality theorem then lies in showing that the categories $\bf deV$ and
${\bf deVBoo}/\!\sim$ are equivalent (Theorem \ref{newdeVries}).

Fedorchuk's duality theorem \cite{F} circumvents the inconvenience of the composition law of the category $\bf deV$ (caused by the non-functoriality of projective covers), by admitting only special morphisms in both categories of the de Vries duality (while keeping the objects), namely so-called quasi-open maps in $\bf KHaus$ and sup-preserving Boolean homomorphisms reflecting the contact relation in $\bf deV$. In Section 5 we briefly indicate how to obtain a proof of this duality theorem following the categorical framework but refer to \cite{DDT} for all details.

As first proved in \cite{D-diss} and recorded here as Proposition \ref{Clust as hom}, just like ultrafilters in a Boolean algebra $A$ may equivalently be described as $\sf 2$-valued Boolean homomorphisms on $A$, clusters in a de Vries algebra $A$ may equivalently be presented by $\sf 2$-valued de Vries morphisms on $A$. Having recognized the functor ${\sf Clust}:{\bf deV}^{\rm op}\to \Set$ just like ${\sf Ult}: {\bf Boo}^{\rm op}\to\Set$ as representable by $\sf 2$, with ${\bf Set}(-,{\sf 2})$ serving as a left adjoint in both cases, one easily sees that both adjunctions induce the ultrafilter monad on $\Set$. In Theorem \ref{hom de Vries} we conclude  that the de Vries dual equivalence as given by the lifted hom-functor
$$\DHC(-,\2)\cong {\sf Clust}:\DHC^{\rm op}\lra\HC$$
is therefore nothing but the comparison functor of ${\bf deV}^{\rm op}$ into the Eilenberg-Moore category of the ultrafilter monad, {\em i.e.}, into $\bf KHaus$ (\cite{Manes, MacLane}). From a categorical perspective, this result makes the de Vries duality look very natural indeed.

Similarly to our proof of the de Vries duality, in Section 7 we give a two-step procedure for the establishment of the BMO duality. First, with Corollary \ref{duality refl} applied to the de Vries dual equivalence and the category $\bf KHaus$ as embedded into $\bf Tych$, we obtain the category $\bf UdeV$
and its dual equivalence with the category $\bf Tych$ (Theorem \ref{fdttychth}). The $\bf UdeV$-objects are de Vries algebras $A$ that come equipped with a subset
 $Y$
 of (the compact Hausdorff space)
${\bf deV}(A,{\sf 2})$ which may then serve as the Stone-\v{C}ech compactification of its subspace $Y$. The second step towards establishing the BMO duality (see \cite{BMO}, and also \cite{Bezh,
BBSV,BMO1}) then consists of employing the Tarski duality to encode the subset $Y$ by its power set, treated as a complete atomic Boolean algebra. Briefly, rather than subsets $Y$ one may equivalently consider de Vries embeddings $A\to B$ into complete atomic Boolean algebras $B$ and arrive at the above-mentioned category $\bf UBdeV$, which turns out to be equivalent to $\bf UdeV$ (Theorem \ref{BMO theorem}).

\section{Brief review: Stone, Tarski, de Vries, Fedorchuk}
\begin{nist}\label{Stone}
\rm
{\em The Stone duality}.
We start by considering the {\em basic dual adjunction}
\begin{center}
$\xymatrix{{\bf Boo}^{\rm op}\ar@/^0.6pc/[rr]^{{\bf Boo}(-,{\sf 2})} & {\scriptstyle \top} & {\bf Set}\ar@/^0.6pc/[ll]^{{\bf Set}(-,{\sf 2})}
}$
\end{center}
of contravariant hom-functors represented by {\sf 2}, considered as both, a two-element Boolean algebra and a mere doubleton set. Under this adjunction, for any Boolean algebra $A$ and any set $X$, the Boolean homomorphisms $\p: A\to{\bf Set}(X,{\sf 2})$ correspond bijectively to the maps $f:X\to{\bf Boo}(A,{\sf 2})$, via $f(x)(a)=\p(a)(x)$ for all $a\in A,\, x\in X$. Of course, ${\bf Set}(-,{\sf 2})$ represents the contravariant powerset functor ${\sf P}$, and ${\bf Boo}(A,2)$ may equivalently be described as the set ${\sf Ult}(A)$ of ultrafilters in $A$, {\em i.e.}, as the set of maximal, non-empty, down-directed, and upwards-closed proper subsets of $A$; on morphisms, both contravariant functors operate by taking inverse images. The bijective correspondence between homomorphisms $\p: A\to{\sf P}X$ and maps $f:X\to{\sf Ult}(A)$ is now described by $(a\in f(x)\iff x\in\p(a))$ for all $a\in A,\, x\in X$. The identity map on ${\sf Ult}(A)$ corresponds to the co-unit of the above adjunction, which is therefore described by
$$\varepsilon_A:A\longrightarrow{\sf P}({\sf Ult}(A)),\quad a\longmapsto \{\mathfrak u\in{\sf Ult}(A)\,|\,a\in \mathfrak u\}.$$
The adjunction induces the ultrafilter monad on $\bf Set$, with underlying endofunctor $X\mapsto \beta X={\sf Ult}({\sf P}(X))$. Its Eilenberg-Moore category is known to be (equivalent to) the category ${\bf KHaus}$ of compact Hausdorff spaces (see \cite{Manes, MacLane}). Therefore, the functor ${\sf Ult}$ can assume the role of the comparison functor from ${\bf Boo}^{\rm op}$ into the E-M category of $\beta$ and, as such, becomes ${\bf KHaus}$-valued: one takes the sets $\varepsilon_A(a),\, a\in A,$ as the basic open (and closed) sets of a compact Hausdorff topology on the set ${\sf Ult}(A)$ that, having a base for {\em clopen} (= closed and open) sets, is {\em zero-dimensional} and, thus, makes ${\sf Ult}(A)$ a {\em Stone space}. The ``lifted" functor $\sf Ult$ is easily seen to still have a left adjoint; it assigns to a compact Hausdorff space $X$ the Boolean algebra ${\sf CO}(X)$ of its clopen subsets:
\begin{center}
$\xymatrix{{\bf Boo}^{\rm op}\ar@/^0.6pc/[rr]^{\sf Ult} & {\scriptstyle \top} & {\bf KHaus}\ar@/^0.6pc/[ll]^{\sf CO}
}$
\end{center}
The unit of this adjunction at a compact Hausdorff space $X$ is the map
$$\eta_X: X\longrightarrow{\sf Ult}({\sf CO}(X)),\quad x\longmapsto \{M\in{\sf CO}(X)\,|\,x\in M\},$$
which is continuous since $\eta_{X}^{-1}(\varepsilon_{{\sf CO}(X)}(M))=M$ for all $M\in{\sf CO}(X)$. By compactness of $X$, it is also surjective, as a standard argument shows: if we had $\mathfrak u\in{\sf Ult}({\sf CO}(X))$ outside the image of $\eta_X$, every $x\in X$ would have a clopen neighbourhood $M_x\notin {\mathfrak u}$, {\em i.e.}, $X\setminus M_x\in{\mathfrak u}$, which would be possible only if the sets $M_x$ cover $X$; but finitely many of these cannot cover $X$ since, as a member of $\mathfrak u$, the intersection of their complements cannot be empty. Consequently, $\eta_X$ is a homeomorphism if (and, trivally, only if,) $\eta_X$ is injective, and since $X$ is Hausdorff, this happens precisely when ${\sf CO}(X)$ is a base for the topology of $X$, {\em i.e.}, if $X$ is a Stone space. Therefore, by restriction of the previous adjunction, one obtains the Stone dual equivalence (\cite{ST};  see, for example, \cite{Halmos,J,kop89})
\begin{center}
$\xymatrix{{\bf Boo}^{\rm op}\ar@/^0.6pc/[rr]^{\sf Ult} & {\scriptstyle \simeq} & {\bf Stone}\ar@/^0.6pc/[ll]^{\sf CO},
}$
\end{center}
with $\bf Stone$ denoting the obvious full subcategory of $\bf KHaus$.

We note the well-known fact
(\cite{Halmos, kop89})
that, under the Stone duality, {\em complete} Boolean algebras correspond  to compact Hausdorff spaces that are {\em extremally disconnected} (so that the closure of any open set is still open). Hence, in a self-explanatory notation one has the dual equivalence
\begin{center}
$\xymatrix{{\bf CBoo}^{\rm op}\ar@/^0.6pc/[rr]^{\sf Ult\;\;\;} & {\scriptstyle \simeq} & {\bf EKH}\ar@/^0.6pc/[ll]^{\sf CO\;\;\;}.
}$
\end{center}
\end{nist}

\begin{nist}\label{StoneMorphRest}
\rm
{\em Morphism restrictions of the Stone duality.}
For later use we list some useful restrictions of the Stone duality to special types of morphisms. Recall that a continuous map $f:X\to Y$ of topological spaces is  {\em
quasi-open\/} (\cite{MP}) if, for every open set $U$ in $X$, the interior $\int(f(U))$ of its image under $f$ may be empty only if $U$ itself is empty.
As shown in \cite[Corollary 3.2(c)]{D-PMD12} and \cite[Corollary 2.4(c)]{D-a0903-2593}, under the Stone duality the quasi-open maps of Stone spaces correspond precisely to those homomorphisms $\p:A\to B$ of Boolean algebras which preserve all (existing) suprema; that is: for all $D\subseteq A$, the existence of $\bigvee D$ in $A$ implies the existence of  $\bigvee \p(D)$ in $B$ and the equality $\p(\bigvee D)=\bigvee \p(D)$. Hence, in an obvious notation, the Stone duality restricts to the dual equivalence
\begin{center}
$\xymatrix{({\bf Boo}_{\rm sup})^{\rm op}\ar@/^0.6pc/[rr]^{\sf Ult\;} & {\scriptstyle \simeq} & {\bf Stone}_{\rm q-open}\ar@/^0.6pc/[ll]^{\sf CO\;}.
}$
\end{center}
\medskip
A ``merger" of this duality with the last duality of \ref{Stone} gives us the dual equivalence
\begin{center}
$\xymatrix{({\bf CBoo}_{\rm sup})^{\rm op}\ar@/^0.6pc/[rr]^{\sf Ult} & {\scriptstyle \simeq} & {\bf EKH}_{\rm q-open}\ar@/^0.6pc/[ll]^{\sf CO}.
}$
\end{center}
It is well known that a continuous mapping between extremally disconnected compact Hausdorff spaces is quasi-open if, and only if, it is open (for a more general statement see, e.g., \cite[Lemma 2.7]{D2009}). Hence, the category
${\bf EKH}_{\rm q-open}$ may be written more simply as ${\bf EKH}_{\rm open}$.
\end{nist}

\begin{nist}\label{Tarski}
\rm
{\em The Tarski duality.} Returning to the basic dual adjunction of \ref{Stone}, one notes that the contravariant powerset functor $\sf P$  actually takes values in $\bf CABA$, the category of complete {\em atomic} Boolean algebras (in which every element is the join of {\em atoms}, {\em i.e.}, of minimal non-bottom elements) and their {\em complete} ( = sup- and, hence, also inf-preserving) Boolean homomorphisms; this, of course, is a full subcategory of ${\bf CBoo}_{\rm sup}$, as considered in \ref{StoneMorphRest}. Routine exercises in Boolean algebra (see, for example, \cite{vanOosten}) show that the map
$$\t_A: A\longrightarrow{\sf P}({\sf At}(A)),\quad a\longmapsto \{x\in{\sf At}(A)\,|\,x\leq a\}$$
becomes an isomorphism when the complete Boolean algebra $A$ is atomic; here ${\sf At}(A)$ denotes the set of atoms in $A$. In other words, $\bf CABA^{\rm op}$ is actually the essential image of $\bf Set$ under the functor $\sf P:{\bf Set}\to{\bf Boo }^{\rm op}$, and one has the dual equivalence
\begin{center}
$\xymatrix{{\bf CABA}^{\rm op}\ar@/^0.6pc/[rr]^{\quad{\sf At}} & {\scriptstyle \simeq} & {\bf Set},\ar@/^0.6pc/[ll]^{\quad{\sf P}}
}$
\end{center}
also known as the Tarski duality. Explicitly, the functor $\sf At$ sends a complete homomorphism $\p:A\to B$ to the map ${\sf At}(B)\to{\sf At}(A)$ which assigns to an atom $y\in B$ the
atom (as one can show) $\at(\p)(y):=\bigwedge\{a\in A\st y\leq\p(a)\}$ in $A$. In other words, for $\p:A\to B$ in $\bf Boo$ and all $a\in A, y\in {\sf At}(B)$ one has
$$y\leq\p(a)\iff{\sf At}(\p)(y)\leq a.$$
In this way one sees easily that there is a natural bijection
$\vk_A:{\sf At}(A)\to {\bf CABA}(A,{\sf 2})$ which maps $x\in {\sf At}(A)$ to the characteristic function $\xi_x:=\vk_A(x)$ of the set $\uparrow\! x=\{a\in A\st x\leq a\}$, and its inverse $\vk_A\inv$ maps $\xi:A\to{\sf 2}$ to $\bigwedge\{a\in A\st \xi(a)=1\}$. Consequently,
\begin{center}
$\xymatrix{{\bf CABA}^{\rm op}\ar@/^0.6pc/[rr]^{\quad{\bf CABA(-,{\sf 2})}} & {\scriptstyle \simeq} & {\bf Set}\ar@/^0.6pc/[ll]^{\quad{\bf Set}(-,{\sf 2})}
}$
\end{center}
is an equivalent presentation of the Tarski duality.

In Section 7 we will use a hybrid of these two presentations of the Tarski duality and, for every complete atomic Boolean algebra $A$ and every set $X$, consider the natural isomorphisms
$$\tilde{\theta}_A:A\longrightarrow{\sf P}({\bf CABA}(A,{\sf 2})),\quad a\longmapsto\{\xi\in{\bf CABA}(A,{\sf 2})\st \xi(a)=1\},$$
$$\chi_X:X\longrightarrow{\bf CABA}({\sf P}(X),{\sf 2}), \ x\longmapsto[\,\chi_X^x:{\sf P}(X)\to{\sf 2},\,(\chi_X^x(M)=1\iff x\in M)\,].$$
\end{nist}

\begin{nist}\label{deVriesAlgebraic}
\rm
{\em The algebraic side of the de Vries duality.}
A set $F$ in a topological space $X$ is {\em regular closed} (or a {\em closed domain} \cite{E}) if it is the closure of its interior in $X$: $F={\rm cl(int}(F))$. The collection ${\sf RC}(X)$ of all regular closed sets in $X$ becomes a Boolean algebra, with the Boolean operations $\vee,\we,\,^*,0,1$ given by
\begin{align*}
F\vee G &= F\cup G, & F\we G &= \cl(\int(F\cap G)), & F^* &=  \cl(X\stm F), & 0 &=  \emptyset, & 1 &=  X.
\end{align*}
The Boolean algebra ${\sf RC}(X)$ is actually complete, with the infinite joins and meets given by
\begin{align*}
 \bigvee_{i\in I}F_i &=  \cl(\bigcup_{i\in I}F_i)\
=\cl(\bigcup_{i\in I}\int(F_i))=\cl(\int(\bigcup_{i\in I}F_i))\quad \text{and}\quad
\bigwedge_{i\in I}F_i &=  \cl(\int(\bigcap_{i\in I}F_i)).
 \end{align*}
 We note in passing that, under the assignment $F\mapsto \int(F)$, ${\sf RC}(X)$ is isomorphic to the Boolean algebra ${\sf RO}(X)$ of {\em regular open} sets $U$ in $X$,
 {\em i.e.,} of those subsets $U$ of $X$
  for which $U=\int(\cl(U))$.
 With the {\em contact relation}
 $\smallfrown_X$ given by
$$F\smallfrown_X G\iff F\cap G\neq\emptyset,$$
${\sf RC}(X)$ becomes a {\em contact algebra} (see \cite{DV1}), that is: a Boolean algebra $A$ provided with a relation $\smallfrown$ satisfying the conditions (C1-4) below; if $X$ is a normal Hausdorff space, then ${\sf RC}(X)$ becomes even a {\em normal contact algebra} $A$ (see \cite{deV, F} where, however, different names have been used), defined to satisfy also conditions (C5-6):
\begin{enumerate}
\renewcommand{\theenumi}{\ensuremath{{\rm C}\arabic{enumi}}}
 \item $a\smallfrown a$ whenever $a>0$;\label{c1}
    \item $a\smallfrown b$ implies $a>0$ and $b>0$;\label{c2}
     \item $a\smallfrown b$ implies $b\smallfrown a$;\label{c3}
     \item $a\smallfrown(b\vee c)$  \tiff\!  $a\smallfrown b$ or $a\smallfrown c$; \label{c4}
 \item if $a\not\smallfrown b$, then $a\not\smallfrown c$ and $b\not\smallfrown c^*$ for some
$c$;\label{c5}
    \item  if $a<1$, then there exists $b>0$ such that
$b\not\smallfrown a$.  \label{c6}
\end{enumerate}
With the {\em non-tangential inclusion} relation $\ll$ on $A$ defined by
$(a\ll b\iff a\not\smallfrown b^*),$
these conditions may equivalently be stated as

\begin{enumerate}
\renewcommand{\theenumi}{{\rm I\,}\arabic{enumi}}
\item $a \ll b \timplies a \leq b$; \label{di1}
\item  $a\leq b\ll c\leq d$ implies $a\ll d$; \label{di3}
\item $0\ll 0$, and $a\ll b$ implies $b^*\ll a^*$; \label{di7}
\item $a\ll c$ and $b\ll c$ implies $a\vee b\ll c$; \label{di4}
\item if  $a\ll c$, then $a\ll b\ll c$  for some $b$; \label{di5}
\item  if $a> 0$, then $b\ll a$ for some $b>0$; \label{di6}
\end{enumerate}
that is: (C1-4)$\iff$(I\,1-4) and (C1-6)$\iff$(I\,1-6). In terms of $\ll$, the contact relation $\smallfrown$ takes the form $(a\smallfrown b\iff a\not\ll b^*)$.
In $({\sf RC}(X),\smallfrown_X)$, the associated relation $\ll_X$ reads as
$$F\ll_X G\iff F\subseteq \int_X(G).$$
Note that if a contact algebra $(A,\smallfrown)$ satisfies also condition C6, then every $a\in A$ may be written as
$$a=\bigvee\{b\in A\st b\ll a\}.$$
Every Boolean algebra $A$ may be endowed with a largest contact relation, $\smallfrown_{\bullet}$, and a least one, $\smfc$, respectively defined  by
$$a\smf_{\bullet} b\Longleftrightarrow a\neq 0 \neq b \qquad{\rm and}\qquad a\smfc b\Longleftrightarrow a\we b\neq 0$$
for all $a,b\in A$. The non-tangential inclusion relation $\ll_{\circ}$ associated with $\smfc$, is just the order relation of the Boolean algebra $A$, which makes $(A,\smfc)$ even a normal contact algebra. We call this NCA structure on $A$ {\em discrete}.

One is now ready to define the category $\bf deV$ of {\em de Vries algebras}: these are complete normal contact algebras, and a {\em de Vries morphism} $\p:A\to B$ is a map of de Vries algebras satisfying the following conditions for all $a,b\in A$:
\begin{enumerate}
\renewcommand{\theenumi}{{\rm V}\arabic{enumi}}
\item $\p(0)=0$;
\item $\p(a\we b)=\p(a)\we \p(b)$;
\item if $a^{\ast}\ll b$ in $A$, then $\p(a)^{\ast}\ll
\p(b)$ in $B$;
\item $\p(a)=\bigvee\{\p(b)\st b\ll a\}$.
\end{enumerate}
From V1-3 alone one obtains $\p(1)=1,\; \p(a^{\ast})\leq \p(a)^{\ast}$ and $(a\ll b\Lra \p(a)\ll \p(b))$.\\ It is important to note that a de Vries morphism is not necessarily a Boolean homomorphism; conversely, a Boolean homomorphism with $(a\ll b\Lra \p(a)\ll \p(b))$ for all $a, b \in A$ satisfies V1-3, but not necessarily V4, unless it preserves suprema.

The composite $\psi\diamond \p:A\to C$ of $\p$ with $\psi:B\to C$ in the category $\bf deV$ is given by
$$(\psi\diamond\p)(a)=\bv\{(\psi\circ \p)(b)\,|\,b\ll a\}$$
for all $a\in A$. In general, this composition differs from the ordinary composition of maps, but it will coincide with it when $\psi$ preserves suprema. The identitity map on $A$ maintains this role in the category $\bf deV$.
\end{nist}

\begin{nist}\label{deVries}
\rm{\em The de Vries duality.}
Extending the contact relation
$\smallfrown$ on a Boolean algebra $A$ to a relation on
 its ultrafilters
 (which, for brevity, will again be denoted by $\smallfrown$)
 by
$$\mathfrak u\smallfrown\mathfrak v\iff \forall\; c\in\mathfrak u,\; d\in\mathfrak v:c\smallfrown d,$$
one shows (\cite[Lemma 3.5, p. 222]{DV1})
that the contact relation for elements $a,b\in A$ is characterized by its ultrafilter extension, via
$$a\smallfrown b\iff \exists\; \mathfrak{u, v}\in{\sf Ult}(A): a\in\mathfrak u,\; b\in\mathfrak v,\; \mathfrak u\smallfrown\mathfrak v.$$
Furthermore, {\em if $A$ is normal, then \  $\smallfrown$ \ is an equivalence relation on ${\sf Ult}(A)$} (\cite{DV1,DUV}).
In order to relate de Vries algebras with compact Hausdorff spaces, instead of considering ultrafilters directly, one uses the closely related concept of {\em cluster}\ \  in a normal contact algebra $A$; this is a subset $\mathfrak c$ of $A$ satisfying the following conditions for all $a,b\in A$:
\begin{enumerate}
\renewcommand{\theenumi}{cl\,\arabic{enumi}}
\item $\mathfrak c\neq\emptyset$;
\item\label{cl1}  $a,b\in\mathfrak c$ implies $a\smallfrown b$;
\item\label{cl2} $a\vee b\in\mathfrak c$ implies $a\in\mathfrak c$ or
$b\in\mathfrak c$;
\item\label{cl3} if $a\smallfrown b$ for all $b\in\mathfrak c$, then $a\in\mathfrak c$.
\end{enumerate}
As an easy consequence one has the property
$(a\in\mathfrak c,\, a\le b\Longrightarrow b\in\mathfrak c)$.
Proceeding as in the proof of Theorem 5.8 of \cite{NW} one shows that {\em every ultrafilter $\mathfrak u$ in a normal contact algebra $A$ gives the cluster}
$${\mathfrak c}_{\mathfrak u}=\{a\in A\,|\,\forall\, b\in\mathfrak u: a\smallfrown b\},$$
and {\em every cluster $\mathfrak c$ in $A$ comes about this way}, that is: $\mathfrak c={\mathfrak c}_{\mathfrak u}$, for some $\mathfrak u\in{\sf Ult}(A)$; actually, {\em for  every $a\in \mathfrak c$ one has $\mathfrak c={\mathfrak c}_{\mathfrak u}$, for some $\mathfrak u\in{\sf Ult}(A)$ with $a\in\mathfrak u$.} One concludes that $\mathfrak c_{\mathfrak u}$ {\em is the unique cluster containing a given ultrafilter} $\mathfrak u$, and that {\em any two clusters comparable by inclusion must actually be equal.} Most importantly, the
relation \ $\smallfrown$ \ for ultrafilters is characterized by $$\mathfrak u\smallfrown\mathfrak v\iff \mathfrak c_{\mathfrak u}=\mathfrak c_{\mathfrak v}.$$

With ${\sf Clust}(A)$ denoting the set of all clusters in a contact algebra $A$,
%and with $\bf deV$ denoting the category of de Vries algebras and de Vries morphisms,
very similarly to the Stone duality, one can now establish the de Vries dual equivalence
\begin{center}
$\xymatrix{{\bf deV}^{\rm op}\ar@/^0.6pc/[rr]^{{\sf Clust}} & {\scriptstyle \simeq} & {\bf KHaus},\ar@/^0.6pc/[ll]^{{\sf RC}}
}$
\end{center}
as follows. We already saw in \ref{deVriesAlgebraic} that, for a compact Hausdorff space $X$, \,${\sf RC}(X)$ becomes a de Vries algebra, and by assigning to a continuous map $f:X\to Y$ the morphism $${\sf RC}(f):{\sf RC}(Y)\to{\sf RC}(X),\;G\mapsto {\rm cl}(f^{-1}({\rm int}(G)),$$ one obtains the functor $\sf RC$. For a compact Hausdorff space $X$
one has the natural map
$$\s_X: X\to{\sf Clust}({\sf RC}(X)),\quad x\mapsto \{F\in{\sf RC}(X)\,|\,x\in F\},$$
and for a de Vries algebra $A$, one considers the map
$$\tau_A: A\to {\sf RC}({\sf Clust}(A)),\quad a\mapsto \{\mathfrak c\in{\sf Clust}(A)\,|\,a\in\mathfrak c\}.$$
With the {\em the topology on ${\sf Clust}(A)$ to be taken to have the sets $\tau_A(a),\;a\in A,$ as its basic closed sets,} one can then compute the interior ${\rm int}(\tau_A(a))$ as the complement of $\tau_A(a^*)$ in ${\sf Clust}(A)$ (see \cite{deV}).
 Now $\sf Clust$ becomes a contravariant functor when one assigns to a de Vries morphism $\p:A\to B$ the map
$${\sf Clust}(B)\to{\sf Clust}(A),\quad {\mathfrak d}\mapsto \{a\in A\st \forall\, b\in \!A\; (\,b\ll a^* \Longrightarrow \p(b)^*\in\mathfrak d\,)\}.$$
We note that, should the $\bf deV$-morphism $\p$ be a Boolean homomorphism, the map ${\sf Clust}(B)\to{\sf Clust}(A)$ is more succinctly described by $\mathfrak c_{\mathfrak v}\mapsto {\mathfrak c}_{\p^{-1}(\mathfrak v)}$, for all $\mathfrak v\in{\sf Ult}(B)$.
Apart from having worked with regular open sets, rather than with regular closed sets, as we do here, de Vries \cite{deV} showed that the natural maps $\s_X$ and $\tau_A$ become natural isomorphisms and thereby established his duality theorem.

As de Vries \cite{deV} also noted, his duality extends the restricted Stone duality between $\bf CBoo$ and $\bf EKH$ (see the end of \ref{Stone}). Indeed,
for any complete Boolean algebra $A$, the discrete de Vries algebra $(A,\smfc)$ (see \ref{deVriesAlgebraic}) satisfies ${\sf Clust}(A)={\sf Ult}(A)$. In this way one sees that the diagram
$$\xymatrix{{\bf deV}^{\rm op}\ar[r]^{\sf Clust\;\;}\ar@{}[rd]|{\rm{}} & {\bf KHaus}\\
{\bf CBoo}^{\rm op}\ar@{^(->}[u]^{}\ar[r]^{\sf Ult} & {\bf EKH}\ar@{^(->}[u]_{}}$$
commutes.
\end{nist}

\begin{nist}\label{Fedorchuk}
\rm
{\em The Fedorchuk duality.}
The Fedorchuk duality may be obtained from the de Vries duality by keeping the objects in both categories under consideration, but restricting the admissible morphisms. On the algebraic side this is done very naturally, as follows.

We take
as objects of the category $\bf Fed$ all de Vries algebras (= complete normal contact algebras), but as morphisms $\p:A\to B$ only suprema-preserving Boolean homomorphisms reflecting the contact relation, so that $(\p(a)\smallfrown \p(b)\Lra a\smallfrown b)$ or, equivalently, preserving the relation $\ll$, that is $(a\ll b\Lra \p(a)\ll\p(b))$, for all $a,b\in A$. In this way, $\bf Fed$ becomes a non-full subcategory of $\bf deV$, and Fedorchuk \cite{F} proved that, under the de Vries duality, $\bf Fed$ becomes dually equivalent to the category ${\bf KHaus}_{\rm q-open}$ of compact Hausdorff spaces and their quasi-open continuous maps (as defined in \ref{StoneMorphRest}). In other words, one has the commutative diagram
$$\xymatrix{{\bf deV}^{\rm op}\ar[r]^{\sf Clust\;\;}\ar@{}[rd]|{\rm{}} & {\bf KHaus}\\
{\bf Fed}^{\rm op}\ar[u]^{}\ar[r]^{\sf Clust\quad\quad} & {\bf KHaus}_{\rm q-open}\ar[u]_{}}$$
and obtains the dual equivalence
\begin{center}
$\xymatrix{{\bf Fed}^{\rm op}\ar@/^0.6pc/[rr]^{{\sf Clust}\quad\quad} & {\scriptstyle \simeq} & {\bf KHaus}_{\rm q-open}\;.\ar@/^0.6pc/[ll]^{{\sf RC}\quad\quad}
}$
\end{center}
We remark that, unlike in $\bf deV$, the morphism composition in $\bf Fed$ coincides with the set-theoretic composition of maps.
\end{nist}

\begin{nist}\label{summary}
\rm
{\em Summary and a look ahead}.
The cuboid of the following commutative diagram summarizes how the de Vries and Fedorchuk dualities extend certain restricted Stone dualities, and its upper part visualizes the BMO extension of the de Vries duality:
\begin{center}
$\xymatrix{& {\bf UBdeV}^{\rm op}\ar[rr]^{\simeq}& &{\bf Tych}\\
&&&\\
& {\bf deV}^{\rm op}\ar@{^(->}[uu]\ar[rr]^{\simeq}\ar@{}[rrd]|{\rm{}} & & {\bf KHaus}\ar@{^(->}[uu]\\
{\bf Fed}^{\rm op}\ar[ru]\ar[rr]^{\quad\quad\simeq}\ar@{}[rrd]|{\rm{}}  && {\bf KHaus}_{\rm q-open}\ar[ru] &\\
& {\bf CBoo}^{\rm op}\ar@{^(->}[uu]^{}\ar[rr]^{\quad\quad\quad\simeq} & & {\bf EKH}\ar@{^(->}[uu]_{}\\
({\bf CBoo}_{\rm sup})^{\rm op}\ar[ru]\ar@{^(->}[uu]^{}\ar[rr]^{\simeq} & & {\bf EKH}_{\rm{open}}\ar[ru]\ar@{^(->}[uu]_{}}$
\end{center}
Our goal is now to provide a general categorical framework which will allow us to build equivalent substitutes for the categories $\bf deV$ and ${\bf UBdeV}$ (as well as for $\bf Fed$), just based on the bottom and right panels of this diagram. %(In particular, the Fedorchuk duality should be established without recourse to the de Vries duality.)
In the case of the de Vries duality, this means that, using just the restricted Stone duality between $\bf CBoo$ and $\bf EKH$, with the latter category being fully embedded into $\bf KHaus$, the framework should provide us with a category that, by construction, is
 dually equivalent
 to $\bf KHaus$ and fairly easily seen to be also equivalent to $\bf deV$; likewise for $\bf UBdeV$. Our key tool to this end is an abstraction of
 {\em projective covers}, also called {\em absolutes}. For the reader's convenience, we briefly recall them next.
\end{nist}

\begin{nist}\label{projective}
\rm
{\em Projective covers and essential surjections}.
Recall that an object $P$ in a category $\CC$ is {\em projective} if the hom-functor $\CC(P,-):\CC\to{\bf Set}$ preserves epimorphisms; that is:
for every epimorphism $f:X\to Y$ in $\CC$, any morphism $g:P\to Y$ factors as $g=f\circ h$, for some morphism $h:P\to X$. Given an object $X$, a projective object $P$ together with an epimorphism $p:P\to X$ is a {\em
projective cover} of $X$ if
the equality $p\circ t=p$, where $t$ is an endomorphism  of $P$, is possible only if $t$ is an isomorphism.
One sees immediately that $X$ determines $P$ and $p$ up to isomorphism: if $q:Q\to X$ is also a projective cover of $X$, then $q\circ h=p$ for some isomorphism $h:P\to Q$.

Also easily shown is the fact that an epimorphism $p:P\to X$ with $P$ projective is a projective cover of $X$ if, and only if, the epimorphism $p$ is {\em essential}, that is: for every morphism $r:Z\to P$, if $p\circ r$ is an epimorphism, so is $r$. Indeed, assuming $p$ to be a projective cover and considering $r$ such that $p\circ r$ is epic, the projectivity of $P$ gives a morphism $s:P\to Z$ with $p\circ r\circ s=p$, so that $r\circ s$ is an isomorphism by hypothesis and, thus, $r$ is forced to be a (split) epimorphism. Conversely, assuming $p$ to be essential, any endomorphism $t$ with $p\circ t=p$ must be epic, so that the projectivity
of $P$ gives a morphism $s$ with $t\circ s=1_P$; but since $p\circ s= p\circ t\circ s = p$, so that also $s$ must be an epimorphism, $s$, and then also $t$, is actually an isomorphism.

We note that the argumentation remains valid if we relativize the notions of projective object and projective cover, by trading ``epimorphism" everywhere for ``$\EE$-morphism", where $\EE$ may be any class of morphisms in $\CC$ that contains all split epimorphisms (=retractions), but a split monomorphism (=section) lies in $\EE$ only if it is an isomorphism.
We also note the known fact
 (see \cite{AT})
  that, when every object $X$ in $\CC$ admits an $\EE$-projective cover $\pi_X: EX\to X$, with $\EE$ a class of extremal epimorphisms in $\CC$ ({\em i.e.}, of those epimorphisms $p$ that factor through a monomorphism $m$ as $p=m\circ g$ only when $m$ is an isomorphism), then $E$ may be made a functor $\CC\to\CC$ and $\pi$ a natural transformation $E\to{\rm Id}_{\CC}$ {\em only if} every object in $\CC$ is already $\EE$-projective---despite the fact that $EX$ and $\pi_X$ are determined by $X$, up to isomorphism.

In $\CC={\bf KHaus}$, the epimorphisms are extremal and coincide with the surjections. An essential epimorphism $f:X\to Y$ is also called {\em irreducible}, as it is characterized by the
 following two properties: 1) $f(X)=Y$, and 2) for all closed subsets $Z$ of $X$, \,$f(Z)=Y$ is possible only when $Z=X$.

By {\em Gleason's Theorem} \cite{Gle}, $P\in{\bf KHaus}$ is projective if, and only if, $P$ is extremally disconnected. Consequently, any irreducible map $p:P\to X$ in $\bf KHaus$ with $P$ extremally disconnected serves as a projective cover for a given $X$. It is well known
(see, for example, \cite{Wa})
that $P$ may be constructed as the Stone-dual of the complete Boolean algebra ${\sf RC}(X)$, with $p=\pi_X$ the map
$$EX={\sf Ult}({\sf RC}(X))\to X$$
that assigns to an ultrafilter $\mathfrak u$ of regular closed sets of $X$ the only point in $\bigcap\mathfrak u$.

Under the name {\em absolute}, projective covers have been investigated intensively in categories of topological spaces larger than $\bf KHaus$, such as the category of
regular Hausdorff
spaces (see, for example, \cite{ArP, PS, PW}), with the notion of projectivity relativized to {\em perfect} surjections. Of course, in $\bf KHaus$ every map is perfect, that is: a closed map with compact fibres. With the restriction to perfect maps, the characterization of projective covers via extremal disconnectedness and irreducibility remains valid when compact Hausdorff spaces are traded for
regular Hausdorff
 spaces.
 For the definition and properties of absolutes of arbitrary topological spaces,  see the survey paper \cite{PS}. A modern and very efficient presentation appeared in \cite{Rump}.

We will  make essential use of {\em Alexandroff's Theorem} \cite[p. 346]{Alex}
(see also \cite[Theorem (d) (3), p. 455]{PW}),
which one derives easily from Ponomarev's results \cite{Pon1} on irreducible maps:
 For $p:X\lra Y$ a closed irreducible map of topological spaces, the map
  $$\rho_p:\RC(X)\lra \RC(Y),\ \ H\mapsto p(H),$$
  is a Boolean isomorphism, with $\rho_p\inv(K)=\cl_X(p\inv(\int_Y(K)))$, for all $K\in\RC(Y)$.

  Finally, we will also use the well-known fact
(see, e.g., \cite{CNG},
p.271, and, for a proof, \cite{VDDB})
that, when $Y$ is a dense subspace of a topological space $X$, then one has the Boolean isomorphisms $r$ and $e$ that are inverse to each other:
 $$r:\rc(X)\lra \rc(Y),\  F\mapsto F\cap Y,\qquad{\rm and}\qquad e:\rc(Y)\lra\rc(X),\  G\mapsto \cl_X(G).$$
 \end{nist}

\section{A general framework for extending dualities}
In fixing our notation, for a dual adjunction we use the symbolism
\begin{center}
$\xymatrix{{\AA}^{\rm op}\ar@/^0.6pc/[rr]^{T} & {\scriptstyle \ep\;\top\;\eta} & {\XX}\ar@/^0.6pc/[ll]^{S}}$
\end{center}
to indicate that $T$ is right adjoint to $S$ with adjunction {\em units} $\eta_X:X\to TSX$ in $\XX$ and {\em counits} $\ep_A:A\to STA$ in $\AA$; they define the natural transformations $\eta:{\rm Id}_{\XX}\to TS$ and $\ep:{\rm Id}_{\AA}\to ST$ satisfying the {\em triangular identities}\footnote{Of course, formally we should have written $\ep^{\rm op}:ST\to{\rm Id}_{\AA^{\rm op}}$ for the counit of the dual adjunction and listed the triangular equalities as $T\ep^{\rm op}\circ\eta T=1_T,\;\ep^{\rm op}S\circ S\eta=1_S$. But in this paper we will generally suppress the explicit use of the op-formalism for functors and natural transformations, in order to keep the notation simple and maintain the symmetric presentation of the units in $\XX$ and the counits in $\AA$.}
$$T\ep\circ\eta T=1_T\quad{\rm and}\quad S\eta\circ \ep S=1_S.$$
When the dual adjunction is a dual equivalence, so that the units and counits are isomorphisms, we replace $\top$ by $\simeq$ in the symbolic notation, as we have already done so in the previous section.

\begin{defi}\label{lifting}
\rm
For given functors $I:\AA\to\BB$ and $J:\XX\to\YY$, a dual adjunction $\xymatrix{{\BB}^{\rm op}\ar@/^0.6pc/[rr]^{\tilde{T}} & {\scriptstyle \tilde{\ep}\;\top\;\tilde{\eta}} & {\YY}\ar@/^0.6pc/[ll]^{\tilde{S}}}$ is a
{\em a right lifting} of $\xymatrix{{\AA}^{\rm op}\ar@/^0.6pc/[rr]^{T} & {\scriptstyle \ep\;\top\;\eta} & {\XX}\ar@/^0.6pc/[ll]^{S}}$ (along $I$ and $J$) if $I$ and $J$ commute with the right adjoints,  so that $\tilde{T}I=JT$ holds:
\begin{center}
$\xymatrix{\BB^{\rm op}\ar[r]^{\tilde{T}} & \YY\\
            \AA^{\rm op}\ar[r]_{T}\ar[u]^{I} & \XX\;;\ar[u]_{J}}$
            \end{center}
it is a {\em left lifting} if $IS=\tilde{S}J$ holds.
\end{defi}

\begin{pro}\label{lifting pro}
{\rm (1)} For a right lifting as in {\rm \ref{lifting}}, there is a natural transformation $\g:IS\to\tilde{S}J$ satisfying the conditions $\tilde{T}\g\circ\tilde{\eta}J=J\eta$ and $\g T\circ I\ep=\tilde{\ep}I$ and being uniquely determined by each of them; $\g$ is an isomorphism if the given dual adjunctions are both dual equivalences.

{\rm (2)} For a left lifting as in {\rm \ref{lifting}}, there is a natural transformation $\d:JT\to\tilde{T}I$ satisfying the conditions $\tilde{S}\d\circ \tilde{\ep}I=I\ep;$ and $\d S\circ J\eta= \tilde{\eta}J$ and being uniquely determined by each of them;
$\d$ is an isomorphism if the given dual adjunctions are both dual equivalences.
\end{pro}

\begin{proof}
(1) Since for all $B\in|\BB|,Y\in|\YY|$ one  has the natural bijections $\BB(B,\tilde{S}Y)\cong\YY(Y,\tilde{T}B)$, the transformation $\g:IS\to\tilde{S}J$ corresponds by adjunction to the transformation $J\eta:J\to JTS=\tilde{T}IS$; explicitly, one has $\g=\tilde{S}J\eta\circ\tilde{\ep}IS$. For dual equivalences, $\eta$ and $\tilde{\ep}$ are both isomorphisms, so that then also $\g$ is an isomorphism.

(2) follows dually from (1); the explicit formula for $\delta$ is now $\delta= \tilde{T}I\ep\circ\tilde{\eta}JT.$
\end{proof}

Given a dual equivalence
$\xymatrix{{\AA}^{\rm op}\ar@/^0.6pc/[rr]^{T} & {\scriptstyle \ep\;\simeq\;\eta} & {\XX}\ar@/^0.6pc/[ll]^{S}}$
and an embedding $J$ of $\XX$ as a full subcategory of a category $\YY$ that, without loss of generality, is assumed to be an inclusion functor, we wish to give a {\em natural construction} for a category $\BB$ into which $\AA$ may be fully embedded via $I$ and which allows for a right lifting along $I$ and $J$ that renders $\BB$ as dually equivalent to $\YY$. Our construction depends on a given class $\PP$ of morphisms in $\YY$, the role model for which is the class of essential surjections of projective covers of compact Hausdorff spaces, as described in \ref{projective}.
\begin{defi}\label{covering class}
\rm
For a full subcategory $\XX$ of $\YY$, we call a class $\PP$ of morphisms in $\YY$ an $\XX$-{\em covering class} in $\YY$ if it satisfies the following conditions:
\begin{enumerate}
\renewcommand{\theenumi}{P\arabic{enumi}}
\item $\forall X\in|\XX|:1_X\in\PP$;
\item every $\XX$-object $X$ is $\PP$-projective in $\YY$, that is:\\ $\forall\;(p:Y\to Y')\in\PP,\,f:X\to Y'\ \exists\;g:X\to Y:\,p\circ g=f;$
\item $\YY$ has enough $\PP$-projectives in $\XX$, that is:\\
$\forall \,Y\in |\YY|\;\ex\,(p:X\to Y)\in\PP\ :\ X\in |\XX|.$
\end{enumerate}
Without loss of generality, one may assume that the domain of every morphism in $\PP$ lies in $\XX$, and that $\PP$ is closed under precomposition with isomorphisms in $\XX$; indeed, with $\PP$ also the class $\overline{\PP}:=\{p:X\to Y\,|\, X\in |\XX|,\, p\in \PP\circ{\rm Iso}(\YY)\}$ satisfies P1-3.
\end{defi}

\begin{exas}\label{coreflective rem}
\rm
(1) If $ \XX$ is a full coreflective subcategory of $\YY$, one always has an $\XX$-covering class $\PP$ of morphisms in $\YY$. Just take for $\PP$ the class of coreflections of $\YY$-objects into $\XX$, that is: of $\YY$-morphisms $p : X \to Y$ with $X\in |\XX|$ such that every $\YY$-morphism $X'\to Y$ with $X' \in |\XX|$ factors uniquely through $p$.

(2) More generally than in (1), assume that the full embedding $J:\XX\to\YY$ admits a functor $E:\YY\to\XX$ and a natural transformation $\pi:JE\to {\rm Id}_{\YY}$, such that $EJ\cong{\rm Id}_{\XX}$ and $\pi J$ is an isomorphism. Then the class
$$\PP_{\pi}=\{p:X\to Y\;|\;X\in|\XX|,\; \beta\circ p=\pi_Y \text{ for some isomorphism }\beta:EY\to X\}$$
is an $\XX$-covering class in $\YY$.
\end{exas}

\begin{const}\label{const}
\rm
As a precursor to the construction of a category $\BB$ as envisaged before \ref{covering class}, for the given dual equivalence and an $\XX$-covering class $\PP$ in $\YY$, we consider the category
$${\sf C}(\AA,\PP,\XX),$$
defined as follows:

\begin{itemize}

\item objects in ${\sf C}(\AA,\PP,\XX)$
are pairs $(A,p)$ with $A\in|\AA|$ and $p:TA\lra Y$ in the class $\PP$;

\item  morphisms $(\varphi,f):(A,p)\lra(A',p')$ in ${\sf C}(\AA,\PP,\XX)$ are given by
morphisms $\varphi:A\lra A'$ in $\AA$ and  $f:Y'\lra Y$ in $\YY$,
such that  $p\circ T\varphi=f\circ p'$:
\begin{center}
$\xymatrix{TA\ar[d]_{p} & TA'\ar[l]_{T\varphi}\ar[d]^{p'}\\
            Y & Y'\ar[l]^{f}}$
            \end{center}

\item    composition is as in $\AA$ and $\YY$, so that $(\p,f)$ as above gets composed with
$(\p',f'):(A',p')\lra(A'',p'')$ by the horizontal pasting of diagrams, that is,
$$(\p',f')\circ(\p,f)= (\p'\circ\p,f\circ f').$$

\item the identity morphism of a ${\sf C}(\AA,\PP,\XX)$-object $(A,p)$ is the ${\sf C}(\AA,\PP,\XX)$
-morphism $(1_A,1_{\cod(p)})$.
\end{itemize}

\noindent On the hom-sets of ${\sf C}(\AA,\PP,\XX)$
we define a compatible equivalence relation by
$$(\p,f)\sim(\psi,g)\Longleftrightarrow f=g,$$
for all $(\p,f),(\psi,g):(A,p)\lra(A',p').$ We denote the equivalence class of $(\p,f)$ by $[\p,f]$%(or $[\p,f]_{(A,p),(A',p')}$, if clarity demands it)
, and now let $\BB$ be the quotient category
$$\BB={\sf C}(\AA,\PP,\XX)/\!\sim,$$
{\em i.e.,} $|\BB|=|{\sf C}(\AA,\PP,\XX)|$, ${\rm Mor}(\BB)=\{[\p,f]\st (\p,f)\in {\rm Mor}({\sf C}(\AA,\PP,\XX))\}$ and $$[\p\ap,f\ap]\circ [\p,f]=[\p\ap\circ\p,f\circ f\ap].$$

\noindent Thanks to condition P1 one has the functor $I:\AA\lra\BB$, defined by
$$(\p:A\lra A')\mapsto (\;I\p=[\p,T\p]:(A, 1_{TA})\lra(A',1_{TA'})\;),$$
which is easily seen to be a full embedding.
\end{const}

It is now straightforward to establish a dual equivalence of $\BB$ with $\YY$, as follows:

\begin{theorem}\label{more duality}
Let $\XX$ be a full subcategory of \  $\YY$ and $\PP$ an $\XX$-covering class in $\YY$.
Then $\xymatrix{{\AA}^{\rm op}\ar@/^0.6pc/[rr]^{T} & {\scriptstyle \ep\;\simeq\;\eta} & {\XX}\ar@/^0.6pc/[ll]^{S}}$ has a right lifting $\xymatrix{{\BB}^{\rm op}\ar@/^0.6pc/[rr]^{\tilde{T}} & {\scriptstyle \tilde{\ep}\;\simeq\;\tilde{\eta}} & {\YY}\ar@/^0.6pc/[ll]^{\tilde{S}}}$ along $I$ and $J$, with $J:\XX\hookrightarrow \YY$ the inclusion and $\BB$,
 as well as
 the full embedding $I:\AA\to\BB={\sf C}(\AA,\PP,\XX)/\!\sim$, being defined as above. The lifted dual equivalence may be chosen to satisfy $$\tilde{T}\tilde{S}={\rm Id}_{\YY},\quad
\tilde{\eta}=1_{{\rm Id}_{\YY}},\quad \tilde{T}\tilde{\varepsilon}=1_{\tilde{T}},\quad \tilde{\varepsilon}\tilde{S}=1_{\tilde{S}},$$
and the canonical isomorphism $\g:IS\to\tilde{S}J$ then satisfies  $\tilde{T}\gamma=J\eta$ and $\gamma T\circ I\varepsilon=\tilde{\varepsilon} I$.
\begin{center}
$\xymatrix{\YY\ar[r]_{\tilde{S}}\ar@/^1.0pc/[rr]^{\rm Id_{\YY}}\ar@{}[rd]|{\cong} & \BB^{\rm op}\ar[r]_{\tilde{T}} & \YY\\
            \XX\ar[r]_{S}\ar[u]^{J} & \AA^{\rm op}\ar[r]_{T}\ar[u]_{I} & \XX\ar[u]_{J}}$
            \end{center}
\end{theorem}

\begin{proof}
$\tilde{T}$ is given by the projection $[\p,f]\mapsto f$; this trivially gives a faithful functor. With P2 we see that $\tilde{T}$ is full:  given $(A,p:TA\to Y),(A',p:TA'\to Y')\in|\BB|$ and $f:Y'\to Y$ in $\YY$, with P2 one obtains $g:TA'\to TA$ with $p\circ g=f\circ p'$, and then $g$ may be written as $T\p$ with $\p:A\to A'$ in $\AA$ since $T$ is full. To define $\tilde{S}$ on objects, with P3  one chooses for every $Y\in |\YY|$ a morphism $\pi_Y:EY\lra Y$ in $\PP$, with $\pi_X=1_X$ for all $X\in |\XX|$ (according to P1), and then puts $\tilde{S}Y=(SEY,\pi_Y\circ\eta^{-1}_{EY})$. For a morphism $f:Y'\lra Y$ in $\YY$, again, P2 and the fullness of $T$ allow one to choose a morphism $\p_f:SEY\lra SEY'$ in $\AA$ with
$\pi_Y\circ\eta^{-1}_{EY}\circ T\p_f=f\circ\pi_{Y'}\circ\eta^{-1}_{EY'}$; we then put
$\tilde{S}f=[\p_f,f]$. Checking the functoriality of $\tilde{S}$ and the identity $\tilde{T}\tilde{S}={\rm Id}_{\YY}$ is straightforward.

For $(A,p:TA\lra Y)\in|\BB|$ one puts $\tilde{\varepsilon}_{(A,p)}=[\varphi_{(A,p)},1_Y]$, with any $\AA$-morphism $\p_{(A,p)}:A\lra SEY$ satisfying $p\circ T\p_{(A,p)}=\pi_Y\circ\eta^{-1}_{EY}$. Clearly, $\tilde{\varepsilon}$ is, like $\tilde{\eta}=1_{{\rm Id}_{\YY}}$, a natural isomorphism satisfying the claimed identities. By Proposition \ref{lifting pro}(1),
the canonical  natural isomorphism $\gamma:IS\to\tilde{S}J$ must necessarily satisfy $\tilde{T}\gamma=I\eta$ and  $\gamma T\circ I\varepsilon=\tilde{\varepsilon} I$; it is given by $\gamma_X= [1_{SX},\eta_X]$ for all $X\in |\XX|$.
\end{proof}

Let us now look at the special case of $\XX$ being coreflective in $\YY$, so that we have an adjunction $J\dashv E:\YY\to\XX$ with counit $\pi:JE\to{\rm Id}_{\YY}$. For simplicity, and without loss of generality, we may assume that the coreflector $E$ and the counit $\pi$ have been chosen to satisfy $EJ={\rm Id}_{\XX}$ and $\pi J=1_J$. As observed in Remark \ref{coreflective rem}, we can then take the class $\PP=\PP_{\pi}$ to be given by the coreflections $\pi_Y:JEY\to Y,\;Y\in\YY,$ precomposed by any $\XX$-isomorphisms, and apply Construction \ref{const}. The equivalence relation $\sim$ on ${\sf C}(\AA,\PP_{\pi},\XX)$ considered there becomes discrete in this case, so that one may simply consider $\BB={\sf C}(\AA,\PP_{\pi},\XX)$ and the full embedding
$$I:\AA\to\BB,\;(\p:A\lra A')\mapsto (\;I\p=(\p,T\p):(A, 1_{TA})\lra(A',1_{TA'})\;),$$
when applying Theorem \ref{more duality} in this situation. We now obtain the following corollary:

\begin{cor}\label{duality corefl}
Let $J:\XX\hookrightarrow\YY$ be coreflective, with $J\dashv E$ and counit $\pi:JE\to{\rm Id}_{\YY}$.
Then $\xymatrix{{\AA}^{\rm op}\ar@/^0.6pc/[rr]^{T} & {\scriptstyle \ep\;\simeq\;\eta} & {\XX}\ar@/^0.6pc/[ll]^{S}}$ has a right lifting $\xymatrix{{\BB}^{\rm op}\ar@/^0.6pc/[rr]^{\tilde{T}} & {\scriptstyle \tilde{\ep}\;\simeq\;\tilde{\eta}} & {\YY}\ar@/^0.6pc/[ll]^{\tilde{S}}}$ along $I$ and $J$, satisfying
$$\tilde{T}\tilde{S}={\rm Id}_{\YY},\quad
\tilde{\eta}=1_{{\rm Id}_{\YY}},\quad \tilde{T}\tilde{\varepsilon}=1_{\tilde{T}},\quad \tilde{\varepsilon}\tilde{S}=1_{\tilde{S}};$$
here the full embedding
$I:\AA\to\BB={\sf C}(\AA,\PP_{\pi},\XX)$ has a reflector $D\dashv I$ with unit $\rho:{\rm Id}_{\BB}\to ID$, satisfying
$DI={\rm Id}_{\AA},\;\rho I=1_{I}$ and making the diagram
\begin{center}
$\xymatrix{\BB^{\rm op}\ar[d]_{D} & \YY\ar[d]^{E}\ar[l]_{\tilde{S}}\\
            \AA^{\rm op} & \XX\ar[l]^{S}}$
            \end{center}
commute, so that $\xymatrix{{\AA}^{\rm op}\ar@/^0.6pc/[rr]^{T} & {\scriptstyle \ep\;\simeq\;\eta} & {\XX}\ar@/^0.6pc/[ll]^{S}}$ becomes a left lifting of
$\xymatrix{{\BB}^{\rm op}\ar@/^0.6pc/[rr]^{\tilde{T}} & {\scriptstyle \tilde{\ep}\;\simeq\;\tilde{\eta}} & {\YY}\ar@/^0.6pc/[ll]^{\tilde{S}}}$ along $D$ and $E$.
In addition to the canonical isomorphism $\g:IS\to\tilde{S}J$, there is therefore the canonical isomorphism $\beta:E\tilde{T}\to TD$, determined by each of the conditions
$\beta\tilde{S}=\eta E$ and $S\beta\circ \ep D=D\tilde{\ep}$; furthermore, $\beta$ and $\g$ connect $\pi$ and $\rho$ via
$$ \tilde{T}\rho\circ J\beta=\pi\tilde{T},\quad \g E\circ \rho \tilde{S}=\tilde{S}\pi.$$
\end{cor}

\begin{proof}
Assuming, without loss of generality, that $EX=X$ and $\pi_X=1_X$ holds for all $X\in|\XX|$, and writing $(\p,f)$ instead of $[\p,f]$ everywhere in the proof of Theorem \ref{more duality}, one may proceed {\em verbatim} as in that proof, in order to define the lifted dual equivalence $\tilde{S}\dashv\tilde{T}$. (Note that, in the definitions of $\tilde{S}$ and $\tilde{\ep}$, the morphisms $\p_f$ and $\p_{(A,p)}$ are now uniquely determined.) For $D$ one simply takes the projection
$$D:\BB\to\AA,\quad ((\p,f):(A,p)\to(A',p'))\mapsto(\p:A\to A'),$$
so that $DI={\rm Id}_{\AA}$ and $D\tilde{S}=SE$ hold trivially. By the last identity, just as $\g:IS\to\tilde{S}I$ corresponds by adjunction to $J\eta$, one obtains with Proposition \ref{lifting pro} the natural isomorphism $\beta:E\tilde{T}\to TD$ corresponding to $D\tilde{\ep}$ by adjunction and, thus, satisfying the stated identities $\beta\tilde{S}=\eta E$ and $S\beta\circ \ep D=D\tilde{\ep}$. Explicitly, for $(A,p:TA\to Y)\in|\BB|$, the (iso)morphism $\beta_{(A,p)}:EY\to TA$ is the only one with $p\circ\beta_{(A,p)}=\pi_Y$.
One easily confirms that
$$\rho_{(A,p)}=(1_A,p): (A,p)\to(A,1_{TA})=ID(A,p)$$
is an $I$-universal arrow for $(A,p)$ and therefore serves as a unit $\rho:{\rm Id}_{\BB}\to ID$ for the adjunction $D\dashv I$. Now the identity $p\circ\beta_{(A,p)}=\pi_Y$ means $\tilde{T}\rho\circ J\beta=\pi \tilde{T}$. It is also straightforward to confirm the remaining identities $\beta\tilde{S}=\eta E$ and $\g E\circ\rho\tilde{S}=\tilde{S}\pi.$
\end{proof}

\begin{rem}
\rm
As shown in Theorem 3.4 of \cite{DDT}, our Corollary \ref{duality corefl} remains intact in the more general context of Example \ref{coreflective rem}(2), that is, when the full subcategory $\XX$ of $\YY$ admits a pseudo-retraction $E:\YY\to\XX$ and a natural transformation $\pi:JE\to{\rm Id}_{\YY}$ such that $\pi J$ is an isomorphism, provided that $\pi$ has the additional property that $\pi_Y\circ\alpha=\pi_Y$ for an automorphism $\alpha:EY\to EY$ is possible only for $\alpha=1_{EY}$. (In this case, $E$ is just {\em semi-right adjoint} to $J$ in the sense of \cite{medvedev}, that is: generally, only one of the two triangular identities required for an adjunction is assumed to hold.) Under this generalization of Corollary \ref{duality corefl}, one has to return to the construction of the category $\BB$ as given in the proof of Theorem \ref{more duality}, but has the advantage that $E$ and $\pi$ facilitate (not a unique, but still) a functorial choice of a representative in all equivalence classes $[\p,f]$ that form the morphisms of $\BB$, and it then suffices to operate with these representatives.
But as we are not aware of any relevant example that would benefit from this  generalization of Corollary \ref{duality corefl}, we skip any further elaboration of it in this paper.
\end{rem}

We now indicate how to dualize Theorem \ref{more duality} and Corollary \ref{duality corefl}, giving an explicit formulation only in the case of Corollary \ref{duality corefl}, as it will be used in Section 7. In our symbolism, noting that we put the ``op" sign only for categories, the dualization of a dual adjunction
$\xymatrix{{\AA}^{\rm op}\ar@/^0.6pc/[rr]^{T} & {\scriptstyle \ep\;\top\;\eta} & {\XX}\ar@/^0.6pc/[ll]^{S}}$
gives us
$\xymatrix{{\AA}\ar@/^0.6pc/[rr]^{T} & {\scriptstyle \eta\;\bot\;\ep} & {\XX}^{\rm op}\ar@/^0.6pc/[ll]^{S}}.$
Hence, $T$ is now considered as left adjoint to $S$, and the roles of $\eta$ and $\ep$ as unit and counit have been interchanged. A dual equivalence may then be written as
$$\xymatrix{{\AA}\ar@/^0.6pc/[rr]^{T} & {\scriptstyle \eta\;\simeq\;\ep} & {\XX}^{\rm op}\ar@/^0.6pc/[ll]^{S}}.$$
Under this dualization, a right lifting along $I,J$ (as defined in \ref{lifting}) becomes a left lifting along $I,J$, and conversely.

Instead of a full coreflective emdedding $J:\XX\hookrightarrow\YY$, we now have to assume that $\XX$ be reflective in $\YY$, with $E:\YY\to \XX$ being left adjoint to $J$. Since in our primary example the reflection morphisms will be embeddings, we denote the unit of the adjunction $E\dashv J$ by $\iota:{\rm Id}_{\YY}\to JE$, instead of $\pi$ as used in the dual situation. Likewise, in dualizing the class $\PP_{\pi}$ and, thus, switching from a notion of projectivity to a notion of {\em injectivity}, we now use the notation
$\JJ_{\iota}$ for the class of all reflection morphisms $\iota_Y:Y\to EY\;(Y\in|\YY|),$ post-composed by any isomorphisms in $\XX$.

\begin{const}\label{dual construction}
\rm
From the class $\JJ_{\iota}$ as above one builds the category
$${\sf D}(\AA,\JJ_{\iota},\XX),$$
dually to ${\sf C}(\AA,\PP_{\pi},\XX)$, as follows:
\begin{itemize}

\item objects
are pairs $(A,j)$ with $A\in|\AA|$ and $j:Y\to TA$ in the class $\JJ_{\iota}$;

\item  morphisms $(\varphi,f):(A,j)\lra(A',j')$ are given by
morphisms $\varphi:A\lra A'$ in $\AA$ and  $f:Y'\lra Y$ in $\YY$
with  $T\varphi\circ j'=j\circ f$:
\begin{center}
$\xymatrix{TA & TA'\ar[l]_{T\varphi}\\
            Y\ar[u]^{j} & Y'\ar[l]^{f}\ar[u]_{j'}}$
            \end{center}

\item    composition in ${\sf D}(\AA,\JJ_{\iota},\XX)$ proceeds by the horizontal pasting of diagrams;

\item the identity morphism of a ${\sf D}(\AA,\JJ_{\iota},\XX)$-object $(A,j)$ is $(1_A,1_{\dom(j)})$.
\end{itemize}
As in the dual situation, there is a full embedding
$$I:\AA\to\BB={\sf D}(\AA,\JJ_{\pi},\XX) ,\;(\p:A\ra A')\mapsto (\;I\p=(\p,T\p):(A, 1_{TA})\ra(A',1_{TA'})\;).$$
\end{const}

The dualization of Corollary \ref{duality corefl} now reads as follows:

\begin{cor}\label{duality refl}
Let $J:\XX\hookrightarrow\YY$ be reflective, with $E\dashv J$ and unit $\iota:{\rm Id}_{\YY}\to JE$.
Then $\xymatrix{{\AA}\ar@/^0.6pc/[rr]^{T} & {\scriptstyle \eta\;\simeq\;\ep} & {\XX}^{\rm op}\ar@/^0.6pc/[ll]^{S}}$ has a left lifting $\xymatrix{{\BB}\ar@/^0.6pc/[rr]^{\tilde{T}} & {\scriptstyle \tilde{\eta}\;\simeq\;\tilde{\ep}} & {\YY}^{\rm op}\ar@/^0.6pc/[ll]^{\tilde{S}}}$ along $I$ and $J$ satisfying
 $$\tilde{T}\tilde{S}={\rm Id}_{\YY},\quad
\tilde{\eta}=1_{{\rm Id}_{\YY}},\quad \tilde{T}\tilde{\varepsilon}=1_{\tilde{T}},\quad \tilde{\varepsilon}\tilde{S}=1_{\tilde{S}};$$
here the full embedding
$I:\AA\to\BB={\sf D}(\AA,\JJ_{\iota},\XX)$ has a coreflector $D$ with counit $\rho:ID\to{\rm Id}_{\BB}$, satisfying
$DI={\rm Id}_{\AA},\;\rho I=1_{I}$ and making the diagram
\begin{center}
$\xymatrix{\BB\ar[d]_{D} & \YY^{\rm op}\ar[d]^{E}\ar[l]_{\tilde{S}}\\
            \AA & \XX^{\rm op}\ar[l]^{S}}$
            \end{center}
commute, so that $\xymatrix{{\AA}\ar@/^0.6pc/[rr]^{T} & {\scriptstyle \eta\;\simeq\;\ep} & {\XX}^{\rm op}\ar@/^0.6pc/[ll]^{S}}$ is a right lifting of
$\xymatrix{{\BB}\ar@/^0.6pc/[rr]^{\tilde{T}} & {\scriptstyle \tilde{\eta}\;\simeq\;\tilde{\ep}} & {\YY}^{\rm op}\ar@/^0.6pc/[ll]^{\tilde{S}}}$ along $D$ and $E$.
In addition to the canonical isomorphism $\g:IS\to\tilde{S}J$, there is therefore the canonical isomorphism $\beta:E\tilde{T}\to TD$, determined by each of the conditions
$\beta\tilde{S}=\eta E$ and $S\beta\circ \ep D=D\tilde{\ep}$; furthermore, $\beta$ and $\g$ connect $\iota$ and $\rho$ via
$$ J\beta\circ \iota\tilde{T}=\tilde{T}\rho,\quad \rho\tilde{S}=\tilde{S}\iota\circ\g E.$$
\end{cor}

For completeness, we list the ``pointwise" definitions of $\tilde{T},D,\tilde{S}, \tilde{\ep},\rho, \g,\beta$ appearing in Corollary \ref{duality refl}, as derived by dual adjustment of those given in Corollary \ref{duality corefl}:
\begin{itemize}
\item $\tilde{T}: ((\p,f):(A,j:Y\to TA)\to(A',j':Y'\to TA'))\mapsto (f:Y'\to Y)$;
\item $D: ((\p,f):(A,j)\to(A',j'))\mapsto (\p:A\to A')$;
\item $\tilde{S}: Y\mapsto (SEY,\eta_{EY}\circ\iota_Y),\;(f\!:\!Y'\!\to \!Y)\mapsto (\p_f,f)$, with $T\!\p_f\circ\eta_{EY'}\circ\iota_{Y'}=\eta_{EY}\circ\iota_Y\circ f$;
\item $\tilde{\ep}_{(A,j:Y\to TA)}=(\p_{(A,j)},1_Y)$, with $\p_{(A,j)}:A\to SEY$ satisfying $T\!\p_{(A,j)}\circ\eta_{EY}\circ\iota_Y=j$;
\item $\rho_{(A,j)}=(A,1_{TA})$;
\; $\g_X=(1_{SX},\eta_X)$;
\; $\beta_{(A,j:Y\to TA)}$ is determined by $\beta_{(A,j)}\circ j= \iota_Y$.
\end{itemize}

\section{A new approach to the de Vries duality}
As mentioned in \ref{projective}, with Gleason's Theorem one sees easily that the class $\PP$ of irreducible continuous maps of compact Hausdorff spaces with extremally disconnected domain is an {\bf EKH}-covering class in {\bf KHaus} (as defined in \ref{covering class}). An application of Theorem \ref{more duality} to the restricted Stone duality
\begin{center}
$\xymatrix{\AA^{\rm op}={\bf CBoo}^{\rm op}\ar@/^0.6pc/[rr]^{\quad T\,=\,{\sf Ult}} & {\scriptstyle \simeq} & {\bf EKH}=\XX\ar@/^0.6pc/[ll]^{\quad S\,=\,{\sf CO}}
}$
\end{center}
(see \ref{Stone}) produces the following dual representation of {\bf KHaus}:

\begin{pro}\label{Fed application}
$\xymatrix{\BB^{\rm op}:=({\sf C}(\AA,\PP,\XX)/\!\sim)^{\rm op}\ar@/^0.6pc/[rr]^{\quad\quad\tilde{T}} & {\scriptstyle \simeq} & \quad{{\bf KHaus}=:\YY}\ar@/^0.6pc/[ll]^{\quad\quad\tilde{S}}.
}$
\end{pro}
\noindent Here the objects of the category ${\sf C}(\AA,\PP,\XX)$ are pairs $(A,p\!:\!{\sf Ult}(A)\to Y)$ with a complete Boolean algebra $A$ and $p\in\PP$; morphisms $(\p,f):(A,p)\to(A',p')$ are given by Boolean homomorphisms $\p:A\to A'$ and maps $f$ that make the diagram
\begin{center}
$\xymatrix{{\sf Ult}(A)\ar[d]_{p} & {\sf Ult}(A')\ar[l]_{{\sf Ult}(\varphi)}\ar[d]^{p'}\\
            Y & Y'\ar[l]^{f}}$
            \end{center}
commute (see Construction \ref{const}). We note that such a map $f$ is necessarily continuous and uniquely determined by $\p$ since, as a continuous surjection of compact Hausdorff spaces, the map $p'$ in $\PP$ provides its codomain with the quotient topology of its domain. The projection functor ${\sf C}(\AA,\PP,\XX)^{\rm op}\to {\sf KHaus},\;(\p,f)\mapsto f,$ induces the compatible relation $\sim$ on ${\sf C}(\AA,\PP,\XX)$, so that
$((\p,f)\sim(\psi,g)\Longleftrightarrow f=g)$
for $(\p,f),(\psi,g):(A,p)\to(A',p')$.
 We obtain the quotient category $\BB$, with the same objects as in ${\sf C}(\AA,\PP,\XX)$. The contravariant functor  $\tilde{T}$ is induced by the projection functor; that is:
 $$\tilde{T}:\BB^{\rm op}\to{\bf KHaus},\;([\p,f]:(A,p)\to(A',p'))\mapsto f.$$
 With $\pi_Y: EY={\sf Ult}({\sf RC}(Y))\to Y$ denoting the Gleason cover of a compact Hausdorff space $Y$ (see \ref{projective}), the adjoint $\tilde{S}$ of $\tilde{T}$ as constructed in Theorem \ref{more duality} formally assigns to $Y$ the $\BB$-object $({\sf CO}(EY), \pi_Y\circ\eta\inv_{EY})$ which, however, is naturally isomorphic to $({\sf RC}(Y),\pi_Y)$; here $\eta$ and $\ep$ are as in \ref{Stone}:
 \begin{center}
 $\xymatrix{EY={\sf Ult(RC}(Y))\ar[dd]_{\pi_Y} & & {\sf Ult(CO}(EY))\ar[d]^{\eta\inv_{EY}}\ar[ll]_{\quad{\sf Ult}(\ep_{{\sf RC}(Y)})}\\
& & EY\ar[d]^{\pi_Y}  \\
 Y & & Y\ar[ll]_{1_Y}
 	}$
 \end{center}
 We may therefore assume $\tilde{S}(Y)=({\sf RC}(Y),\pi_Y)$
 and $\tilde{S}(f)=[\p_f,f]$ for $f:Y\ap\to Y$ in ${\bf KHaus}$, where $\p_f:{\sf RC}(Y)\to {\sf RC}(Y\ap)$ is a Boolean homomorphism such that
 $\pi_Y\circ {\sf Ult}(\p_f)=f\circ\pi_{Y\ap}$.
 This
  leaves all assertions of Theorem \ref{more duality} in tact. In fact, it simplifies them since we now have that the natural isomorphism $\gamma$ of Theorem \ref{more duality} is actually an identity transformation: $IS=\tilde{S}J$.

  In order to show that {\bf KHaus} is dually equivalent to the category {\bf deV}
(see \ref{deVries}),
we just need to exhibit {\bf deV} as equivalent to the category $\BB$ of Proposition \ref{Fed application}.
We start by proving a technical lemma which generalizes \cite[Proposition 1.5.4]{deV}.

  \begin{defi}\label{devriestransf}
\rm
For complete normal contact (= de Vries) algebras $A$ and $A'$ and a monotone function $\varphi:A\to A'$, we call the function
 $\p^{\vee}:A\lra A'$, defined by
 $$\p^{\vee}(a):= \bigvee\{\p(b)\st b\ll a\}$$
 for every $a\in A$, the {\em de Vries transform} of $\p$.
\end{defi}
With the pointwise order of functions of ordered sets, one notes immediately that
	(1) $\p^{\vee}$ is monotone; (2) $\p^{\vee}\leq \p$; (3) if $\p\leq\tilde{\p}$, then $\p^{\vee}\leq\tilde{\p}^{\vee}$. Less trivially one has:

\begin{lm}\label{mainlm2}
Let $A$, $A'$, $A''$ be de Vries algebras. If the monotone function $\p:A\to A'$ preserves the relation $\ll$,
then one has the implication
$$b\ll a\;\Longrightarrow\; \p(b)\ll\p^{\vee}(a)$$
for all $a,b\in A$; furthermore, for every monotone function  $\psi:A'\lra A''$ one obtains
$$(\psi\circ\p)^{\vee}=(\psi^{\vee}\circ  \p^{\vee})^{\vee}.$$
\end{lm}

\begin{proof}
For the first assertion, given $b\ll a$ in $A$, one picks $c\in A$ with $b\ll c\ll a$ and obtains
$\p(a)\ll\p(c)\leq\p^{\vee}(b)$, by the $\ll$-preservation of $\p$ and the definition of $\p^{\vee}(b)$.

For the second assertion, exploiting the properties (1-3) above, one first observes $\psi^{\vee}\circ\p^{\vee}\leq\psi\circ \p$ and concludes $(\psi^{\vee}\circ\p^{\vee})^{\vee}\leq(\psi\circ \p)^{\vee}$.
To prove ``$\geq$'', considering any $b\ll a$ in $A$, we again pick $c\in A$ with $b\ll c\ll a$ and obtain with the first assertion $\p(b)\ll\p^{\vee}(c)$ which, by definition of $\psi^{\vee}(\p^{\vee}(c))$ and $(\psi^{\vee}\circ\p^{\vee})^{\vee}(a)$, gives us
$$(\psi\circ\p)(b)=\psi(\p(b))\leq\psi^{\vee}(\p^{\vee}(c))=(\psi^{\vee}\circ\p^{\vee})(c)\leq(\psi^{\vee}\circ\p^{\vee})^{\vee}(a).$$
Thus, taking the join over all $b\ll a$, we see $(\psi\circ \p)^{\vee}(a)\leq (\psi^{\vee}\circ\p^{\vee})^{\vee}(a)$, as desired.
\end{proof}

One routinely verifies that the composition in the category $\bf deV$ of de Vries algebras, given by (see \ref{deVriesAlgebraic})
$$\psi\diamond\p=(\psi\circ\p)^{\vee},$$  stays in the range of de Vries morphisms. Now we form the auxiliary category
$${\bf deVBoo}$$
whose objects are the same as those of {\bf deV}, {\em i.e.}, are complete normal contact algebras; a morphism $\p: A\to A'$ in {\bf deVBoo} is a {\em Boolean} homomorphism reflecting the contact relation or, equivalently, preserving the associated relation $\ll$ (see \ref{deVriesAlgebraic}). Unlike $\diamond$ in $\bf deV$, the composition of ${\bf deVBoo}$-morphisms proceeds by ordinary map composition. With the map $\varepsilon_A:A\to{\sf CO(Ult}(A))$ defined as in \ref{Stone}, we first show:

\begin{pro}\label{UV}
There are functors
\begin{center}
$\xymatrix{{\sf C}(\AA,\PP,\XX)\ar[r]^{\;\;U} & {\bf deVBoo}\ar[r]^{V} & {\bf deV},}$
\end{center}
where $V$ maps objects identically and $U$ provides the complete Boolean algebra $A$ belonging to an object $(A,p)$ in ${\sf C}(\AA,\PP,\XX)$ with the normal contact relation
$$a\smallfrown_p b\iff p(\varepsilon_A(a))\cap p(\varepsilon_A(b))\neq \emptyset\iff\exists\,\mathfrak u, \mathfrak v\in{\sf Ult(A)}:a\in\mathfrak u,b\in\mathfrak v,p(\mathfrak u)=p(\mathfrak v).$$
\end{pro}

\begin{proof} First we confirm that $\smallfrown_p$ as defined makes $A$ a complete normal contact algebra, for $(A,p:{\sf Ult}(A)\to Y)\in |{\sf C}(\AA,\PP,\XX)|$. Indeed, by Alexandroff's Theorem (see \ref{projective}), the map $\rho_p: {\sf RC(Ult}(A))\to {\sf RC}(Y),\;H\mapsto p(H),$
is a Boolean isomorphism, and so is $\varepsilon_A:A\to{\sf CO(Ult}(A))={\sf RC(Ult}(A))$, by the Stone duality. Hence, $A$ is, like ${\sf RC}(Y)$, a complete normal contact algebra, since the contact relation $\smallfrown_p$ has been obtained by transferring the contact relation of ${\sf RC}(Y)$ along the (inverse of the) isomorphism $\rho_p\circ\varepsilon_A$.

Defining $U$ on morphisms by $[(\p,f):(A,p)\to(A',p')]\longmapsto \p$, we must just show that $\p:A\to A'$ reflects the contact relations imposed by $U$. That, however, is obvious: if $\p(a)\smallfrown\p(b)$ in $A'$, then $p'(\mathfrak u')=p'(\mathfrak v')$ for some $\mathfrak u', \mathfrak v'\in{\sf Ult}(A')$ with $\p(a)\in\mathfrak u', \p(b)\in\mathfrak v'$; consequently,
$p(\p^{-1}(\mathfrak u'))=f(p'(\mathfrak u'))=f(p'(\mathfrak v'))=p(\p^{-1}(\mathfrak v')),$
which implies $a\smallfrown b$.

We define $V$ on a morphism $\p:A\to B$ in {\bf deVBoo} by
$V\p:=\p^{\vee}$
and must first confirm that the map $\p^{\vee}:A\to B$ satisfies conditions (V1-4) of \ref{deVriesAlgebraic}.
It is straightforward to show that, since $\p$ satisfies (V2), the map
 $\p^{\vee}$ satisfies (V2) and (V4).
  Condition (V1) holds trivially for $\p^{\vee}$. Hence, we are left with having to confirm that $\p^{\vee}$ satisfies (V3).
 Given $a,b\in A$ and $a^*\ll b$, we find $c,d\in A$ such that $a^*\ll c\ll d\ll b$. Since $\p$ preserves the Boolean negation and the relation $\ll$, we have

\medskip

\begin{tabular}{lll}
$(\p^{\vee}(a))^*$&$=(\bigvee\{\p(e)\st e\in A, e\ll a\})^*$ \\
&$=\bigwedge\;\{\p(e)^*\st e\in A,  e\ll a\}$ \\
&$=\bigwedge\;\{\p(e^*)\st e\in A,  e\ll a\}$ \\
&$=\bigwedge\;\{\p(e)\st e\in A,  a^*\ll e\}\;\le \p(c)\ll \p(d)\,\le \p^{\vee}(b)$.\\
 \end{tabular}

 \medskip

One has $a=\bigvee\{b\in A\st b\ll a\}$ for all $a\in A$, so that $V$ preserves the identity map on $A$.
 The preservation by $V$ of the composition follows from Lemma \ref{mainlm2}.
 \end{proof}

\begin{nist}\label{pupv}
\rm Note that if, $(A,p:{\sf Ult}(A)\to Y)\in |{\sf C}(\AA,\PP,\XX)|$, then for all $\mathfrak{u},\mathfrak{v}\in{\sf Ult}(A)$,
$$\mathfrak{u}\smallfrown_p \mathfrak{v}\Longleftrightarrow p(\mathfrak{u})=p(\mathfrak{v}).$$
Indeed, the implication ``$\Longrightarrow$'' follows easily from the Hausdorffness of $Y$ and the fact that $\{\ep_A(a)\st a\in A\}$ is an open base for ${\sf Ult}(A)$, while  the converse implication is obvious.
\end{nist}

\begin{pro}\label{U equivalence}
The functor $U$ has a right inverse $W:{\bf deVBoo}\to {\sf C}(\AA,\PP,\XX)$ with $W\circ U\cong \rm{Id}$. In particular, $U$ is an equivalence of categories.	
\end{pro}

\begin{proof}
With the contact relation $\smallfrown$ of an object $A$ in ${\bf deVBoo}$ extended to an equivalence relation on ${\sf Ult}(A)$ as in \ref{deVries}, one puts
$$WA=(A,\,\pi_A:{\sf Ult}(A)\to {\sf Ult}(A)/\!\smallfrown),$$
where $\pi_A$ is the canonical projection
 That $\pi_A$ is an irreducible continuous map in ${\bf KHaus}$ with extremally disconnected domain and, thus, belongs to $\PP$, has been confirmed in Lemma 4.4 of \cite{DDT}. Hence, $WA$ is indeed a ${\sf C}(\AA,\PP,\XX)$-object; moreover, as $\pi_A$
  induces the original contact relation on $A$ (see Proposition \ref{UV} and \ref{deVries}), one has $U(WA)=A$.

For a morphism $\p:A\to A' $ in $\bf deVBoo$, reflection of $\smallfrown$ by $\p$ implies preservation of $\smallfrown$ by ${\sf Ult}(\p)$; indeed, as shown in Proposition 4.6 of \cite{DDT}, this follows quite easily from the Hausdorffness of $Y={\sf Ult}(A)/\!\smallfrown$. Hence, with $Y'={\sf Ult}(A')/\!\smallfrown$ one obtains a uniquely determined continuous map $f_{\p}$ making the diagram
\begin{center}
$\xymatrix{{\sf Ult}(A)\ar[d]_{\pi_A} & {\sf Ult}(A')\ar[l]_{{\sf Ult}(\varphi)}\ar[d]^{\pi_{A'}}\\
            Y & Y'\ar[l]_{f_{\p}}}$
            \end{center}
commute; explicitly,
$f_{\p}(\pi_{A'}(\mathfrak u'))=\pi_A(\p\inv(\mathfrak u'))$, for all $\mathfrak u'\in{\sf Ult}(A')).$
Obviously then, with $W\p=(\p,f_{\p})$, we obtain a well-defined functor $W$. Since, trivially, $U(W\p)=\p$ for all $\p$, one has $U\circ W={\rm Id}_{\bf deVBoo}$.

Finally, for $(A,p)\in|{\sf C}(\AA,\PP,\XX)|$, $U$ provides $A$ with the contact relation $\smallfrown_p$, and then $W$ provides $A$ with the projection $\pi_A$. We may define $$\la_{(A,p)}: W(U(A,p))=(A,\pi_A)\to(A,p),\quad\la_{(A,p)}=(1_A,\ell_p),$$
with $\ell_p$ determined by the commutative diagram
\begin{center}
$\xymatrix{{\sf Ult}(A)\ar[d]_{\pi_A} & {\sf Ult}(A)\ar[l]_{{\sf Ult}(1_A)}\ar[d]^{p}\\
           {\sf Ult}A)/\!\smallfrown_{p}  & Y\ar[l]_{\quad\quad\ell_{p}}}$
            \end{center}
Quite trivially, $\ell_p$ is a homeomorphism, making $\lambda_{(A,p)}$ an isomorphism in ${\sf C}(\AA,\PP,\XX)$, which is obviously natural in $(A,p)$.
\end{proof}

Unlike $U$, the functor $V$ of Proposition \ref{UV} is {\em not} an equivalence of categories. But it induces an equivalence relation on $\bf deVBoo$ that is compatible with the equivalence relation $\sim$ on ${\sf C}(\AA,\PP,\XX)$, as considered in Proposition \ref{Fed application}. Concretely:

\begin{lm}\label{UV2}
For all morphisms $(\p,f),(\psi,g):(A,p)\to(A',p')$ in ${\sf C}(\AA,\PP,\XX)$, one has $f=g$ if, and only if, $V\p=V\psi$; that is: $(f=g\iff \p^{\vee}=\psi^{\vee})$.
\end{lm}

\begin{proof}
Under the hypothesis $f=g$ one has $p\circ{\sf Ult}(\p)=f\circ p'=g\circ p'=p\circ{\sf Ult}(\psi)$, which means $\p\inv(\mathfrak u')\smallfrown_p\psi\inv(\mathfrak u')$ for all $\mathfrak u'\in {\sf Ult}(A')$. To show $\p^{\vee}(a)\leq\psi^{\vee}(a)$ for $a\in A$, it suffices to prove that for all $b\ll a$ there is some $c\ll a$ with $\p(b)\leq\psi(c)$. Assuming the opposite and, for the then given $b\ll a$, picking some $c\in A$ with $b\ll c\ll a$, one has $\p(b)\not\leq\psi(c)$. This implies $d:=\p(b)\wedge\psi(c^*)>0$, so that we find $\mathfrak u'\in {\sf Ult}(A')$ containing $d$. Consequently, $b\in\p\inv({\mathfrak u'})$ and $c^*\in\psi\inv({\mathfrak u'})$. But then $b\smallfrown_p c^*$, which means $b\not\ll c$, a contradiction. Since $\psi^{\vee}(a)\leq\p^{\vee}(a)$ follows by symmetry, the identity $\p^{\vee}=\psi^{\vee}$ is thereby confirmed.

Conversely, having $\p^{\vee}=\psi^{\vee}$, since $p'$ is surjective, it suffices to show $\p\inv(\mathfrak u')\smallfrown_p\psi\inv(\mathfrak u')$ for all $\mathfrak u'\in {\sf Ult}(A')$ to conclude $f=g$
 (see \ref{pupv}).
 Assuming $\p\inv(\mathfrak u')\not\smallfrown_p\psi\inv(\mathfrak u')$ for some $\mathfrak u'\in {\sf Ult}(A')$ we obtain $b\in\p\inv(\mathfrak u')$ and $a\in A$ with  $a^*\in\psi\inv(\mathfrak u')$ and $b\not\smallfrown_p a^*$, which means $\p(b),\psi(a^*)\in\mathfrak u'$ and $b\ll a$. Consequently, with the monotonicity of $\psi$ one has
$$0<\p(b)\wedge\psi(a)^*\leq\p(b)\wedge(\bigvee_{c\ll a}\psi(c))^*,$$
which implies $\p(b)\not\leq\bigvee_{c\ll a}\psi(c)=(V\psi)(a)$. But since $\p(b)\leq \p^{\vee}(a)$, this contradicts the hypothesis $\p^{\vee}=\psi^{\vee}$.
\end{proof}

\begin{const}\label{quotientfunctors}
\rm
With the compatible equivalence relation $\sim$ on ${\sf C}(\AA,\PP,\XX)$ as defined in Proposition \ref{Fed application}, and with $\backsim$ denoting the equivalence relation on {\bf deVBoo} induced by $V$ (so that $\p\backsim\psi\iff \p^{\vee}=\psi^{\vee}$), the assertion of Lemma \ref{UV2} reads as ($(\p,f)\sim(\psi,g)\iff\p\backsim\psi) $. Denoting by $\langle\p\rangle$ the $\backsim$-equivalence class of a morphism $\p$ in {\bf deVBoo}, the functors $U$ and $V$ therefore induce faithful functors
$$\BB=\xymatrix{{\sf C}(\AA,\PP,\XX)/\!\sim\ar[r]^{\quad\overline{U}} & {\bf deVBoo}/\!\backsim\ar[r]^{\quad\overline{V}} & {\bf deV}},$$
mapping objects like $U$ and $V$, and morphisms by $[\p,f]\mapsto \langle\p\rangle$ and $\langle\p\rangle\mapsto \p^{\vee}$, respectively.
\end{const}

To obtain the de Vries duality, it now suffices to prove that
the functors $\overline{U}$ and $\overline{V}$ are equivalences of categories. The first of these two assertions comes almost for free now, and it actually produces a new duality for the category $\bf KHaus$ that is of independent interest:

\begin{theorem}\label{modified deVries} {\rm (Modified de Vries Duality Theorem)}
The functor $\overline{U}$ is an equivalence of categories. As a consequence, $\bf KHaus$ is dually equivalent to the category
$\bf deVBoo/\!\!\backsim$,
 whose objects are complete normal contact algebras and whose morphisms are equivalence classes of Boolean morphisms reflecting the contact relations, to be composed by ordinary map composition of their representatives.
\end{theorem}

\begin{proof}
Just as $U$ induces the functor $\overline{U}$, the functor $W$ of Proposition \ref{U equivalence} induces the functor $$\overline{W}:{\bf deVBoo}/\!\backsim\,\longrightarrow {\sf C}(\AA,\PP,\XX)/\!\sim,\quad \langle \p\rangle\mapsto [\p,f_{\p}].$$
It maps objects as $W$ does and still satisfies $\overline{U}\circ\overline{W}={\rm Id}$ and $\overline{W}\circ\overline{U}\cong{\rm Id}$. Hence, $\overline{U}$ is an equivalence of categories, with quasi-inverse $\overline{W}$. The claimed duality now follows with Proposition \ref{Fed application}.
\end{proof}

Since $\overline{V}$ maps objects identically and is trivially faithful, it suffices to show that $\overline{V}$ is full, in order for us to conclude that it is an isomorphism of categories. To this end, we first confirm the following proposition in which the assertions (2) and (3) are
due to de Vries \cite{deV},
but our proofs are new. Here $\ep_A$ and $\tau_A$ are defined as in \ref{Stone} and \ref{deVries}, and $\pi_A:{\sf Ult}(A)\to{\sf Ult}(A)/\!\smallfrown$ is the irreducible projection of Proposition \ref{U equivalence}.

\begin{pro}\label{full prop} {\rm(1)} For every de Vries algebra $(A,\smallfrown)$, one has the homeomorphism
$$\gamma_A:{\sf Ult}(A)/\!\smallfrown\;\lra {\sf Clust}(A),\;[\mathfrak u]\longmapsto\mathfrak c_{\mathfrak u}=\{a\in A\st\forall b\in\mathfrak u:a\smallfrown b\}$$
(making ${\sf Clust}(A)$ a compact Hausdorff space),
and the Boolean isomorphism
$$\rho_{\gamma_A\circ\pi_A}:{\sf CO}({\sf Ult}(A))\lra{\sf RC}({\sf Clust}(A)),\;H\longmapsto\gamma_A(\pi_A(H)),$$
which fits into the commutative diagram
\begin{center}
$\xymatrix{& A\ar[ld]_{\ep_A}\ar[dr]^{\tau_A} &\\
{\sf CO}({\sf Ult}(A))\ar[rr]^{\rho_{\gamma_A\circ\pi_A}} & & {\sf RC}({\sf Clust}(A))\,.}$
\end{center}

\bigskip

{\rm (2)} $\tau_A$ is a Boolean isomorphism which preserves the relation $\ll$ and, thus, is a {\bf deV}-isomorphism.

\bigskip

{\rm (3)}
For every morphism $\alpha:A\to A'$ in $\bf deV$, one has the continuous map
$$\hat{\alpha}:{\sf Clust}(A')\to{\sf Clust}(A),\quad {\mathfrak c'}\mapsto \{a\in A\st \forall\, b\in \!A\; (\,b\ll a^* \Longrightarrow \alpha(b)^*\in\mathfrak c'\,)\},$$
which fits into the commutative diagram
\begin{center}
$\xymatrix{
A\ar[rr]^{\a}\ar[d]_{\tau_A} & & A'\ar[d]^{\tau_{A'}}\\
{\sf RC}({\sf Clust}(A))\ar[rr]^{{\sf RC}(\hat{\a})} & & {\sf RC}({\sf Clust}(A'))\,.\\
}$	
\end{center}
\end{pro}

\begin{proof} (1) We first note that the family $\{\tau_A(a)\st a\in A\}$ can indeed be taken as a closed base for the space $\Clustsf(A)$
since $\tau_A(0)=\ems$, and for all $a,b\in A$ one has
$$\tau_A(a\vee b)=\{\mathfrak c\in\Clustsf(A)\st a\vee b\in\mathfrak c\}=\{\mathfrak c\in\Clustsf(A)\st a\in\mathfrak c\mbox{ or }b\in\mathfrak c\}=\tau_A(a)\cup\tau_A(b).$$

 From the assertions of \ref{deVries} we know that $\gamma_A$ is a well-defined bijective map.
Also, by applying Alexandroff's Theorem (see \ref{projective}) to $\pi_A:{\sf Ult}(A)\to Y ={\sf Ult}(A)/\!\smallfrown$, we obtain the Boolean isomorphism
$$\rho_{\pi_A}:{\sf CO}({\sf Ult}(A))={\sf RC}({\sf Ult}(A))\to {\sf RC}(Y),\quad H\mapsto \pi_A(H).$$
That $\gamma_A$ is actually a homeomorphism
 follows with
$$\gamma_A(\pi_A(\ep_A(a)))=\{{\mathfrak c}_{\mathfrak u}\st a\in\mathfrak u\in {\sf Ult}(A)\}=\{\mathfrak c\st a\in\mathfrak c\in{\sf Clust}(A)\}=\tau_A(a),$$
for all $a\in A$. Consequently,
$$\rho_{\gamma_A}:{\sf RC}(Y)\to{\sf RC}({\sf Clust}(A)),\quad K\mapsto \gamma_A(K),$$
and then also $\rho_{\gamma_A\circ\pi_A}=\rho_{\gamma_A}\circ\rho_{\pi_A}$, are Boolean isomorphisms, and the triangle of (1) commutes.
Finally, since ${\sf Ult}(A)/\!\smallfrown\;$ is a compact Hausdorff space (as shown in  \cite[Lemma 4.4]{DDT}) and $\g_A$ is a homeomorphism, ${\sf Clust}(A)$ is also a compact Hausdorff space.

(2) Clearly, by (1), $\tau_A$ is a Booolean isomorphism. With the assertions of \ref{deVries} (see also \cite[Corollary 3.4, p.222]{DV1}), for all $a,b\in A$ one has

\medskip

\begin{tabular}{ll}
$ a\smallfrown b$ & $\iff \exists\, \mathfrak u,\mathfrak v\in{\sf Ult}(A): a\in\mathfrak u,\, b\in\mathfrak v,\,\mathfrak u\smallfrown \mathfrak v$\\
 &$\iff \exists\, \mathfrak u,\mathfrak v\in{\sf Ult}(A): a\in\mathfrak u,\,b\in\mathfrak v,\,\mathfrak c_{\mathfrak u}=\mathfrak c_{\mathfrak v} $ \\
 &$\iff \exists\,\mathfrak c\in{\sf Clust}(A): a,b\in\mathfrak c$\\
 & $\iff \tau_A(a)\smallfrown_Z\tau_A(b),$\\
  \end{tabular}

  \medskip

 \noindent where $Z:=\Clustsf(A)$. Consequently, $\tau_A$ preserves the relation $\ll$. %Then $\tau_A$ is a {\bf deV}-isomorphism.

 \medskip

(3) We must first confirm that $\hat{\alpha}(\mathfrak c')$ is a cluster of $A$, for every $\mathfrak c'\in{\sf Clust}(A')$. Obviously,
$1\in\hat{\alpha}(\mathfrak c')$,
so that (cl\,1) is satisfied. For (cl\,2), we consider $a,b\in\hat{\alpha}(\mathfrak c')$ and assume that we had $a\not\smallfrown b$, that is: $a\ll b^*$. Choosing $c,d\in A$ such that $a\ll c\ll d^*\ll b^*$ or, equivalently, $b\ll d\ll c^* \ll a^*$, the definition of $\hat{\alpha}(\mathfrak c')$ gives $(\a(c^*))^*\in\mathfrak c'$ and $(\a(d^*))^*\in\mathfrak c'$, so that $(\a(c^*))^*\smallfrown(\a(d^*))^*$. However, since $c\ll d^*$, (V3) gives $(\a(c^*))^*\ll \a(d^*)$, which means $(\a(d^*))^*\not\smallfrown (\a(c^*))^*$, a contradiction.

To confirm (cl\,3), we consider $a,b\in A$ with $a\vee b\in\a(\mathfrak c')$ and assume we had
 $a\nin\hat{\a}(\mathfrak c')$ and $b\nin\hat{\a}(\mathfrak c')$. Then we obtain $c,d\in A$ such that $c\ll a^*$, $d\ll b^*$ and $(\a(c))^*\nin\mathfrak c'$,
$(\a(d))^*\nin\mathfrak c'$. From $c\we d\ll a^*\we b^*=(a\vee b)^*$ we conclude
$$(\a(c\we d))^*=(\a(c)\we\a(d))^*=(\a(c))^*\vee(\a(d))^*\in\mathfrak c',$$ which implies  $(\a(c))^*\in\mathfrak c'$ or $(\a(d))^*\in\mathfrak c'$, a contradiction.

Finally, for (cl\,4), we consider $a\in A$ with $a\smallfrown b$ for every $b\in\hat{\a}(\mathfrak c')$ and assume  $a\nin\hat{\a}(\mathfrak c')$, so that, for some $c\in A$, one has $c\ll a^*$, but $(\a(c))^*\nin\mathfrak c'$.
 Then
$\a(c)\in\mathfrak c'$. We claim that $c$ lies in $\hat{\a}(\mathfrak c')$. Indeed, if $d\in A$ satisfies $d\ll c^*$, then, as  $\a$ preserves $\ll$, we have $\a(d)\ll\a(c^*)$; therefore, $\a(c)\le(\a(c^*))^*\ll (\a(d))^*$ which, with $\a(c)\in\mathfrak c'$, implies $(\a(d))^*\in\mathfrak c'$;
thus $c\in\hat{\a}(\mathfrak c')$.
 But, since $c\ll a^*$, we have $a\not\smallfrown c$, despite $c\in\hat{\a}(\mathfrak c')$ -- a contradiction.

To show the continuity of $\hat{\a}$, we first observe that, since $\{\tau_A(a)\st a\in A\}={\sf RC}(Z)$ is a base for closed sets in $Z={\sf Clust}(A)$, the set $${\sf RO}(Z)=\{{\rm int}(F)\st F\in{\sf RC}(Z)\}=\{{\rm int}(\tau_A(a))\st a\in A\}$$
is a base for open sets in $Z$. We must therefore show that $\hat{\a}\inv({\rm int}(\tau_A(a)))$ is open in $Z'={\sf Clust}(A')$, for every $a\in A$. Indeed, denoting the Boolean operations of ${\sf RC}(Z)$ as in \ref{deVriesAlgebraic}, with $F= \tau_A(a)$ we have $$\int(F)=Z\stm\cl(Z\stm F)=Z\stm F^*=Z\stm \tau_A(a^*)=\{\mathfrak c\in Z\st a^*\nin\mathfrak c\}.$$
Hence, for every $\mathfrak c'\in Z'$, one obtains
$$\mathfrak c'\in\hat{\a}\inv({\rm int}(F))\iff a^*\nin \hat{\a}(\mathfrak c')\iff\exists\, b\in A\;(b\ll a,\,(\a(b))^*\nin\mathfrak c').$$
Since $(\a(b))^*\nin\mathfrak c'$ means equivalently $ \mathfrak c'\nin\tau_{A'}((\a(b))^*)$, we conclude that the set $$\hat{\a}\inv({\rm int}(F))=
\bigcup_{b\ll a}Z'\stm \tau_{A'}((\a(b))^*)=\bigcup_{b\ll a}\int(\tau_{A\ap}(\a(b)))$$
is indeed open in $Z'$. Furthermore, continuing to take advantage of the Boolean isomorphism $\tau_A$ and the de Vries morphism $\a$, we see that
$$\cl(\hat{\a}\inv({\rm int}(F)))=\cl(\bigcup_{b\ll a}\int(\tau_{A\ap}(\a(b))))
=\bigvee_{b\ll a}\tau_{A\ap}(\a(b))=\tau_{A\ap}(\bigvee_{b\ll a}\a(b))=\tau_{A\ap}(\a(a))$$
holds in $Z'$. But this proves precisely the claimed identity ${\sf RC}(\hat{\a})\circ\tau_A=\tau_{A'}\circ\a$.
\end{proof}	

As a final step in completing the proof of the fullness of the functor $\overline{V}$ we prove the following lemma. Here we understand $T_3$-space (resp., $T_4$-space) to mean a regular (resp., normal) Hausdorff space.

\begin{lm}\label{mainlm}
Let $p:X\to Y,\,p':X'\to Y',\,f:X'\to X,\,g:Y'\to Y$ be continuous maps of topological spaces with $p\circ f=g\circ p'$. If $p$ is closed irreducible and $Y$ a $T_4$-space, then every $G\in\RC(X)$ satisfies the identity

$$\cl_{Y'}(g\inv(\int_Y(p(G))))=\bigvee\{p\ap(f\inv(H))\st H\in\RC(X) \mbox{ and } p(H)\sbe\int_Y(p(G))\}.$$
\end{lm}

\begin{proof}
Alexandroff's Theorem (see \ref{projective}) gives us the Boolean isomorphism $$\rho_p:\RC(X)\lra \RC(Y),\ \ H\mapsto p(H),$$
  whose inverse maps $K\in\RC(Y)$ to $\cl_X(p\inv(\int_Y(K)))$.
  Since $Y$ is a $T_3$-space, for every $G\in\RC(X)$ the theorem implies
  $$\int(p(G))=\bigcup\{p(H)\st H\in\RC(X)\mbox{  and }p(H)\sbe\int(p(G))\}.$$
  In describing the inverse image under $g$ of this union, we first note that, since $p'$ is surjective, we have
   $g\inv (p(H))=p'(p'^{-1} (g\inv(p(H))))=p'(f\inv(p\inv(p(H)) ))$
  for every $H\in\RC(X)$, which gives us
   $$g\inv(\int(p(G)))=\bigcup\{p\ap(f\inv(p\inv(p(H))))\st H\in\RC(X)\mbox{ and }p(H)\sbe\int(p(G))\}.$$
  Since $Y$ is a $T_4$-space, Alexandroff's Theorem gives us for every $H$ with $p(H)\sbe\int(p(G))$ contributing to this union a set $H\ap\in\RC(X)$ with $p(H)\sbe \int(p(H\ap))\sbe p(H\ap)\sbe\int(p(G))$; moreover, $H'=\rho_p\inv(\rho_p(H'))=\cl(p\inv(\int(p(H'))))$. Consequently,
  $H\sbe p\inv(p(H))\sbe p\inv(\int(p(H')))\sbe \cl(p\inv(\int(p(H'))))=H'.$
  As a result,
    $$g\inv(\int(p(G)))=\bigcup\{p\ap(f\inv(H))\st H\in\RC(X)\mbox{ and }p(H)\sbe\int(p(G))\}.$$
  Finally, using  \ref{deVriesAlgebraic}, we conclude the claimed formula.
 \end{proof}

We can now sum up and return to the functor
$$\overline{V}:{\bf deVBoo}/\!\!\backsim\;\longrightarrow{\bf deV}$$
of \ref{quotientfunctors} and complete our alternative proof of the de Vries dual equivalence.

\begin{theorem}\label{newdeVries} {\rm (The de Vries Duality Theorem)}
The functor
$\overline{V}$
is an isomorphism of categories.
Consequently, the category $\bf KHaus$ is dually equivalent to the category $\bf deV$ of de Vries algebras and their morphisms.
\end{theorem}

\begin{proof}
In order to show that $\overline{V}$ is full, given $\alpha:A\to A'$ in $\bf deV$, we must show  $\alpha=V\p$, for some $\p:A\to A'$ in
$\bf deVBoo$. Proposition \ref{full prop} produces the continuous map $\hat{\alpha}={\sf Clust}(\a):{\sf Clust}(A')\to{\sf Clust}(A)$, as well as the homeomorphism
$\gamma_A:{\sf Ult}(A)/\!\smallfrown\;\lra {\sf Clust}(A),\;[\mathfrak u]\mapsto\mathfrak c_{\mathfrak u}$. With $\bar{f}=\gamma_A\inv\circ\hat{\a}\circ\gamma_{A'}$, the
Gleason Theorem (see \ref{projective}) together with the
fact that the map $\pi_A:{\sf Ult}(A)\to{\sf Ult}(A)/\!\smallfrown$ is a surjection
give us a continuous map $f$ making the diagram
\begin{center}
$\xymatrix{{\sf Ult}(A)\ar[d]_{\pi_A} & {\sf Ult}(A')\ar[l]_{f}\ar[d]^{\pi_{A'}}\\
{\sf Ult}(A)/\!\smallfrown\ar[d]_{\gamma_A} & {\sf Ult}(A')/\!\smallfrown\ar[l]_{\bar{f}}\ar[d]^{\gamma_{A'}}\\
{\sf Clust}(A) & {\sf Clust}(A')\ar[l]^{\hat{\alpha}}\\
}$	
\end{center}
commute. By the Stone duality, $f={\sf Ult}(\p)$, for a Boolean homomorphism $\p:A\to A'$.
Obviously,
the commutativity of the above diagram implies that ${\sf Ult}(\p)$ preserves the contact relation for ultrafilters, so that
(by Proposition 4.6 of \cite{DDT})
 $\p$ itself must reflect the contact relation and therefore be a morphism in $\bf deVBoo$.

Apply\-ing Lemma \ref{mainlm} to the outer rectangle of the above diagram, thus
putting $p:=\g_A\circ\pi_A,\, p\ap:=\g_{A\ap}\circ\pi_{A\ap}$ and $g:=\hat{\alpha}$, we obtain
$$\RCsf(\hat{\a})=(\rho_{p\ap}\circ \COsf(\Ultsf(\p))\circ\rho_p\inv)^{\vee}.$$
 Now, Proposition \ref{full prop}(3) implies
 $\tau_{A\ap}\circ\a\circ\tau_A\inv=(\rho_{p\ap}\circ \COsf(\Ultsf(\p))\circ\rho_p\inv)^{\vee}$ and thus
 $$\a=\tau_{A\ap}\inv\circ\z^{\vee}\circ\tau_A,\ \ \mbox{where}\ \  \z:=\rho_{p\ap}\circ \COsf(\Ultsf(\p))\circ\rho_p\inv.$$
 With Lemma \ref{mainlm2} and Proposition \ref{full prop}(2) we then obtain

 \medskip

 \begin{tabular}{ll}
$\a$ & $=\a^{\vee} =((\tau_{A\ap}\inv\circ\z^{\vee})\circ\tau_A)^{\vee}=(((\tau_{A\ap}\inv)^{\vee}\circ\z^{\vee})^{\vee}\circ(\tau_A)^{\vee})^{\vee}$\\ & $=
 ((\tau_{A\ap}\inv\circ\z)^{\vee}\circ(\tau_A)^{\vee})^{\vee} =
 (\tau_{A\ap}\inv\circ\z\circ\tau_A)^{\vee}$\\ & $=(\tau_{A\ap}\inv\circ\rho_{p\ap}\circ \COsf(\Ultsf(\p))\circ\rho_p\inv\circ\tau_A)^{\vee}.$\\
 \end{tabular}

 \medskip
\noindent Finally, applying Proposition \ref{full prop}(1) and the Stone Duality, we conclude  $$\a=(\ep_{A\ap}\inv\circ \COsf(\Ultsf(\p))\circ \ep_A)^{\vee}=\p^{\vee}=V(\p),$$
as desired.
\end{proof}

We also confirm that our constructions leading up to Theorem \ref{newdeVries} give, up to natural isomorphisms, the functors $\sf Clust$ and $\sf RC$ furnishing the classical de Vries dual equivalence, with unit $\sigma$ and counit $\tau$ as described in \ref{deVries}; that is:
\begin{cor}
The diagram
\begin{center}
$\xymatrix{{\bf deV}^{\rm op}\ar@/^0.2pc/[r]^{\overline{V}^{-1}\;\,}\ar@/^2.8pc/[rrr]^{\sf Clust} & ({\bf deVBoo}/\!\backsim)^{\rm op}\ar@/^0.2pc/[r]^{\overline{W}\quad}\ar@/^0.2pc/[l]^{\overline{V}\;\;} & ({\sf C}(\AA,\PP,\XX)/\!\sim)^{\rm op}\ar@/^0.2pc/[r]^{\quad\tilde{T}}\ar@/^0.2pc/[l]^{\overline{U}\quad} & {\bf KHaus}\ar@/^0.2pc/[l]^{\quad\tilde{S}}\ar@/^2.8pc/[lll]^{\sf RC}\\
}$	
\end{center}
commutes in an obvious sense, up to natural isomorphism.	
\end{cor}
\begin{proof}
It suffices to confirm that the functor $\tilde{T}\circ\overline{W}\circ\overline{V}^{-1}$ is naturally isomorphic to $\sf Clust$ since then its left adjoint, $\overline{V}\circ\overline{U}\circ\tilde{S}$, must be naturally isomorphic to $\sf RC$. Indeed, for $(A,\smallfrown)\in |{\bf deV}|$, by Proposition \ref{full prop} one has the homeomorphism
$$\xymatrix{\tilde{T}(\overline{W}(\overline{V}^{-1}(A,\smallfrown)))=\tilde{T}(A,\pi_A)={\sf Ult}(A)/\!\smallfrown\ar[rr]^{\qquad\qquad\qquad\gamma_A} & & {\sf Clust}(A).\\
}$$
Checking that $\gamma_A$ is natural in $A$ involves going back to the morphism definitions of the functors involved, but that is a routine matter.	
\end{proof}

\section{Two pathways to the Fedorchuk duality}

Sections 3 and 4 lead to two novel, but distinct methods of establishing the Fedorchuk duality, one using the de Vries duality, the other {\em not} using it. Under the former method, having established the de Vries duality with the help of the categorical extension technique as in Section 4, one may simply follow Fedorchuk's original approach \cite{F} and obtain his duality by showing that the de Vries duality may be restricted accordingly (as visualized by the top panel of the diagram of \ref{summary}).

Under the latter method, similarly to how we established the de Vries duality in this paper, one applies the categorical construction to the restricted Stone duality
\begin{center}
$\xymatrix{\AA^{\rm op}=({\bf CBoo}_{\rm sup})^{\rm op}\ar@/^0.6pc/[rr]^{\quad T\,=\,\sf Ult} & {\scriptstyle \simeq} & {\bf EKH}_{\rm open}=\XX\ar@/^0.6pc/[ll]^{\quad S\,=\,\sf CO}\;.
}$
\end{center}
As this approach has been outlined in \cite {DDT}, here we give only a brief summary of it.

\medskip

{\em Step 1}. One first shows that the category $\XX$ as above is coreflective in the category of compact Hausdorff spaces and quasi-open maps, $\YY={\bf KHaus}_{\rm q-open}$; see \cite{Rump} or Proposition 4.1 of \cite{DDT}.

{\em Step 2}. Now one can choose $\PP$ to be the class of coreflections (as in Example \ref{coreflective rem}(1)) and apply Corollary \ref{duality corefl} to obtain the dual equivalence
\begin{center}
$\xymatrix{\BB^{\rm op}:={\sf C}(\AA,\PP,\XX)^{\rm op}\ar@/^0.6pc/[rr]^{\tilde{T}} & {\scriptstyle \simeq} & \quad{{\bf KHaus}_{\rm q-open}=\YY}\ar@/^0.6pc/[ll]^{\tilde{S}}\;.
}$
\end{center}

{\em Step 3}. One establishes an equivalence
\begin{center}
$\xymatrix{{\bf Fed}\ar@/^0.6pc/[rr]^{{F}\quad} & {\scriptstyle \simeq} & {\sf C}(\AA,\PP,\XX)\,,\ar@/^0.6pc/[ll]^{{G}\quad}}$
\end{center}
of categories, with the functors $F$ and $G$ defined on objects like the functors $W$ und $U$ of Propositions \ref{U equivalence} and \ref{UV}, respectively.

{\em Step 4}. The composition of the dual equivalence of Step 2 with the dualization of the equivalence of Step 3 produces the Fedorchuk duality, as in

\begin{center}
$\xymatrix{{\bf Fed}^{\rm op}\ar@/^0.4pc/[rr]^{{F}\quad}\ar@/^2.3pc/[rrrr]^{\sf Clust} & {\scriptstyle \simeq} & {\sf C}(\AA,\PP,\XX)^{\rm op}\ar@/^0.4pc/[ll]^{{G}\quad}\ar@/^0.4pc/[rr]^{\tilde{T}\quad} & {\scriptstyle \simeq} & {{\bf KHaus}_{\rm q-open}}\ar@/^0.4pc/[ll]^{\tilde{S}\quad}\ar@/^2.3pc/[llll]^{\sf RC}\;.
}$
	
\end{center}

\section{The hom representation of the de Vries duality}
In \ref{Stone} we sketched the fact that the Stone duality arises naturally from a dual adjunction represented by the Boolean algebra ${\sf 2}$ inducing the ultrafilter monad on ${\bf Set}$, as a restriction of the comparison functor into the Eilenberg-Moore category of the monad, ${\bf KHaus}$. In this Section we show that the de Vries duality may be obtained in exactly the same manner; actually, more succinctly so, as no restriction of the comparison functor is necessary. The key observation, proved first in \cite{D-diss}, is that, as a contravariant $\bf Set$-valued functor, ${\sf Clust}$ is represented by $\sf 2$, now considered as a discrete de Vries algebra.

In more detail,
for every de Vries algebra $A$, we consider the map
$$\omega_A:{\bf deV}(A,{\sf 2})\lra {\sf Clust}(A),\quad \p\longmapsto \{a\in A\st \,\p(a^*)=0\}.$$
\begin{pro}\label{Clust as hom}
 The maps $\omega_A\;(A\in|{\bf deV}|)$ define a natural isomorphism $\omega:{\bf deV}(-,{\sf 2})\lra{\sf Clust}$ of functors \, ${\bf  deV}^{\rm op}\lra {\bf Set}$.	
\end{pro}

\begin{proof}
For every de Vries morphism $\p:A\to{\sf 2}$, we should first confirm that $\omega_A(\p)$ is indeed a cluster in $A$.

(cl\,1) Since $\p(1^*)=\p(0)=0$, one trivially has $1\in\omega_A(\p)\neq\emptyset$.

(cl\,2) For $a,b\in\omega_A(\p)$, let's suppose we had $a\not\smallfrown b$, that is: $a\ll b^*$. Then (V3) implies $(\p(a^*))^*\ll\p(b^*)$, {\em i.e.}, $1\leq 0$, a contradiction.

(cl\,3) In the presence of (V2), from $a\vee b\in \omega_A(\p)$ one obtains $0=\p((a\vee b)^*)=\p(a^*\we b^*)=\p(a^*)\we\p(b^*)$, which gives $a\in\omega_A(\p)$ or $b\in\omega_A(\p)$.

 (cl\,4) For $a\in A$ with $a\smallfrown b$ for all $b\in\omega_A(\p)$, let's suppose we had $a\nin\omega_A(\p)$, so that $\p(a^*)=1$. Then, by (V4), there must be some
 $b\in A$ with $b\ll a^*$ and $\p(b)=1$. Since  $a\not\smallfrown b$, by hypothesis on $a$ we must have $b\nin\omega_A(\p)$, which means $\p(b^*)=1$. But this contradicts the trivial fact $0=\p(b\we b^*)=\p(b)\we\p(b^*)$.

Since, by definition, the cluster $\mathfrak c=\omega_A(\p)$ satisfies
\begin{equation*}\label{phisigma}
\p(a)=0\iff a^*\in\mathfrak c
\end{equation*}
for all $\a\in A$, one can now take this equivalence to, conversely, define a map $\p=\p_{\mathfrak c}:A\to{\sf 2}$ for any given cluster $\mathfrak c$ in $A$. Once we have shown that $\p_{\mathfrak c}$ is a de Vries morphism, it is then clear that $\mathfrak c\mapsto \p_{\mathfrak c}$ is inverse to $\omega_A$, so that $\omega_A$ is bijective.

(V1): From $0^*=1\in\mathfrak c\neq\emptyset$ one obtains  $\p_{\mathfrak c}(0)=0$.

(V2): Since the cluster $\mathfrak c$ is upwards closed, the map $\p_{\mathfrak c}$ is monotone, so that $\p_{\mathfrak c}(a\we b)\leq \p_{\mathfrak c}(a)\we\p_{\mathfrak c}(b)$ follows for all $a,b\in A$.
 Assuming now $\p_{\mathfrak c}(a\we b)=1$, so that $(a\we b)^*\nin\mathfrak c$, we have $a^*\vee b^*\nin\mathfrak c$. This means that neither $a^*$ nor $b^*$ can lie in the upwards closed set $\mathfrak c$. %Now, using (CL3) and (C4), we obtain that $a^*\nin\s$ and $b^*\nin\s$.
 Consequently, $\p_{\mathfrak c}(a)=1=\p_{\mathfrak c}(b)$. We conclude that $\p_{\mathfrak c}(a\we b)=\p_{\mathfrak c}(a)\we \p_{\mathfrak c}(b)$ holds always.

(V3):  Given $a,b\in A$ with $a^*\ll b$, assume first $\p_{\mathfrak c}(a)=0$, so that $a^*\in\s$. Since $a^*\not\smallfrown b^*$, we then obtain that $b^*\nin\mathfrak c$, or $\p_{\mathfrak c}(b)=1$. We conclude that $(\p_{\mathfrak c}(a))^*\ll\p_{\mathfrak c}(b)$ holds, trivially so when $\p_{\mathfrak c}(a)=1$.

(V4): Since $\p_{\mathfrak c}$ is monotone (see (V2)), we certainly have for all $a\in A$ that
$\p_{\mathfrak c}(a)\geq\bigvee\{\p_{\mathfrak c}(b)\st b\in A,\, b\ll a\}$. To show equality, assume $\p_{\mathfrak c}(a)=1$, that is, $a^*\nin\mathfrak c$. Then, by (cl\,4), there exists $c\in\mathfrak c$ with $a^*\not\smallfrown c$, or $c\ll a$. We may then pick some $b\in A$ with  $c\ll b\ll a$. From $c\not\smallfrown b^*$ we now obtain $b^*\nin\mathfrak c$, which means $\p_{\mathfrak c}(b)=1$. As $b\ll a$, this confirms (V4).

We must finally confirm the naturality of $\omega$, that is: for every de Vries morphism $\a:A\to A'$, we have to show the commutativity of the diagram
\begin{center}
$\xymatrix{{\bf deV}(A,{\sf 2})\ar[d]_{\omega_A} && {\bf deV}(A',{\sf 2})\ar[d]^{\omega_{A'}}\ar[ll]_{{\bf deV}(\a,{\sf 2})}\\
{\sf Clust}(A) && {\sf Clust}(A')\ar[ll]^{{\sf Clust}(\a)}\\
}$	
\end{center}
Here the map $\omega_A\circ{\bf deV}(\a,{\sf 2})$ sends every de Vries morphism $\p:A'\to{\sf 2}$ to the cluster $$\omega_A(\p\diamond\a)=\{a\in A\st(\p\circ\a)^{\vee}(a^*)=0\},$$
while the map ${\sf Clust}(\a)\circ\omega_{A'}$ sends $\p$ to the cluster
$$\{a\in A\st\forall\,b\,(b\ll a^*\Rightarrow(\a(b))^*\in\omega_{A'}(\p))\}=\{a\in A\st\forall\,b\,(b\ll a^*\Rightarrow(\p\circ\a)(b)=0)\}.$$
But since $(\p\circ\a)^{\vee}(a^*)=\bigvee\{(\p\circ\a)(b)\st b\in A,\, b\ll a^*\}$, the two clusters coincide.	
\end{proof}

 In strong analogy to the fundamental adjunction
\begin{center}
$\xymatrix{{\bf Boo}^{\rm op}\ar@/^0.6pc/[rr]^{{\bf Boo}(-,{\sf 2})} & {\scriptstyle \top} & {\bf Set}\ar@/^0.6pc/[ll]^{{\bf Set}(-,{\sf 2})}
}$
\end{center}
of \ref{Stone} underlying the Stone duality, we will now set up a dual adjunction, replacing Boolean by de Vries algebras, and then show that it may be used to build up the de Vries duality in a categorical manner, via the comparison functor of ${\bf deV}^{\rm op}$ into the Eilenberg-Moore category of its induced ${\bf Set}$-monad which, just as in the case of the Stone duality \ref{Stone}, turns out to be (up to isomorphism) the category ${\bf KHaus}$.
To this end, in what follows we regard the power set ${\sf P}(X)$ of a set $X$ as a discrete de Vries algebra, by taking $\ll$ to be the inclusion order; that is, ${\bf Set}(X,{\sf 2})$ is regarded as a discrete de Vries algebra with the pointwise order inherited from the two-chain $\sf 2$.
\begin{pro}\label{second fundamental} There is an adjunction
\begin{center}
$\xymatrix{{\bf deV}^{\rm op}\ar@/^0.6pc/[rr]^{{\bf deV}(-,{\sf 2})} & {\scriptstyle \top} & {\bf Set}\ar@/^0.6pc/[ll]^{{\bf Set}(-,{\sf 2})}
}$
\end{center}
whose induced monad on ${\sf Set}$ is (isomorphic to) the ultrafilter monad.	
\end{pro}
\begin{proof}
It suffices to establish, for every set $X$, a bijective correspondence between the de Vries morphisms $\p:A\to{\bf Set}(X,{\sf 2})$ and the maps $f:X\to{\bf deV}(A,{\sf 2})$, naturally so for every de Vries algebra $A$. Such a correspondence may be facilitated by the constraints
$$f(x)(a)=\p(a)(x)$$
for all $x\in X,\,a\in A$, {\em i.e.}, $f(x)(-)=\p(-)(x)$ for all $x\in X$. Indeed, given a de Vries morphism $\p$, the values of the map $f$ defined by these constraints are obviously de Vries morphisms, and conversely.
The naturality of this correspondence is easily established, similarly to the naturality proof of $\omega$, which completes the proof for the claimed adjunction.

To compute the induced monad, by Proposition \ref{Clust as hom} one has, for every set $X$, the natural isomorphisms
$${\bf deV}({\bf Set}(X,{\sf 2}),{\sf 2})\cong {\bf deV}({\sf P}(X),{\sf 2})\cong{\sf Clust}({\sf P}(X))={\sf Ult}({\sf P}(X)).$$
That the last identity holds more generally for all complete Boolean algebras was already mentioned at the end of \ref{deVries}. In the case of the Boolean algebra  ${\sf P}(X)$, it is easy to see that a  cluster $\mathfrak c$ in ${\sf P}(X)$ is indeed an ultrafilter, since  $\mathfrak c$ may be written as $\mathfrak c_{\mathfrak u}=\{F\subseteq X\st \forall\, G\in\mathfrak u\,(F\cap G\neq \emptyset)\}$, for some ultrafilter $\mathfrak u$ on $X$ (see \ref{deVries}). But every set $F\in\mathfrak c_{\mathfrak u}$ must actually lie in $\mathfrak u$, since otherwise we would have its complement lying in $\mathfrak u$ which, by definition of $\mathfrak c_{\mathfrak u}$, is impossible. Hence, $\mathfrak c=\mathfrak c_{\mathfrak u}=\mathfrak u$ is an ultrafilter on $X$. Proving that the resulting identity ${\sf Clust}({\sf P}(X))={\sf Ult}({\sf P}(X))$ extends to set maps involves only another routine check.

Since the bijective correspondence underlying the adjunction is the same as the correspondence of the fundamental adjunction of \ref{Stone} leading to the Stone duality, the pointwise definitions of the units and counits of the current adjunction are the same as the corresponding definitions for the fundamental adjunction. Consequently, the current adjunction induces a monad isomorphic to the monad induced by the fundamental adjunction, {\em i.e.}, the ultrafilter monad.
\end{proof}

Since the category of Eilenberg-Moore algebras of the ultrafilter monad on $\bf Set$ is (isomorphic to) the category ${\bf KHaus}$ (see \cite{Manes},\cite{MacLane},\cite{AHS}), the $\bf Set$-valued functor ${\bf deV}(-,{\sf 2})$ ``lifts'' to become the comparison functor
$${\bf deV}(-,{\sf 2}):{\bf deV}^{\rm op}\lra {\bf KHaus},$$
with the topology on ${\bf deV}(A,{\sf 2})$ to be induced by the counit $\tilde{\tau}_A$ of the adjunction, for every de Vries algebra $A$. But since ${\bf deV}(-,{\sf 2})\cong {\sf Clust}$ by Proposition \ref{Clust as hom}, this counit must arise as a composite of the map $\tau_A$ of \ref{deVries} with the natural bijection induced by $\omega_A$ of Proposition \ref{Clust as hom}, as in
\begin{center}
$\tilde{\tau}_A = \xymatrix{[\;A\ar[rr]^{\tau_A\qquad} & & {\sf RC}({\sf Clust}(A))\ar[rr]^{{\sf RC}(\omega_A)\;\;} & & {\sf RC}({\bf deV}(A,{\sf 2}))\;].}$
\end{center}
As a consequence of the de Vries duality theorem one obtains the following theorem:
\begin{theorem}\label{hom de Vries}
The de Vries dual equivalence may be presented as the comparison functor into the Eilenberg-Moore category of the monad induced by the adjunction of Proposition {\em \ref{second fundamental}}.	
\end{theorem}

\begin{rem}
\rm
(1) We note that here we derived Theorem \ref{hom de Vries} by {\em using} the de Vries duality theorem in a significant way. But our presentation shows that this may be avoided. Indeed, since the category ${\bf deV}$ has equalizers, the general theory of monads describes how to obtain a left adjoint to the comparison functor ${\bf deV}(-,{\sf 2}):{\bf deV}^{\rm op}\lra {\bf KHaus}$ induced by the adjunction of Proposition \ref{second fundamental}, providing also an explicit description of the units and counits of the ``lifted'' adjunction. The core of the proof of the duality theorem then rests on the need to establish that these units and counits are all isomorphisms, for which one may use some of the key tools provided in Section 4.

(2) To the benefit of readers not familiar with basic monad theory and the Eilenberg-Moore construction, here is a more elaborate description of the functors and natural isomorphisms appearing in Theorem \ref{hom de Vries}, as they emerge from the categorical setting.

Transporting, for every $A\in|\DHC|$, the topological structure of ${\sf Clust}(A)$ onto $\DHC(A,\2)$, via the bijection $\om_A$, one obtains the family $\{\om_A\inv(\tau_A(a))\st a\in A\}$
 as a base for closed sets for the topology on $\DHC(A,\2)$ under which the map $\om_A$ becomes a homeomorphism. Then, for every $\a\in\DHC(A,A\ap)$, the map $$\DHC(\a,\2):\DHC(A\ap,\2)\lra\DHC(A,\2),\;\p'\longmapsto \p'\diamond\a,$$ is continuous, by the commutativity of the diagram in the proof of Proposition \ref{Clust as hom}. This defines the functor
 $$\DHC(-,\2):\DHC^{\rm op}\lra\HC,$$
 which, qua definition, is isomorphic to the de Vries equivalence $\sf Clust$ and, thus, produces the {\em hom version of the de Vries duality}.
 The counits
$$\tilde{\tau}_A=\rc(\om_A)\circ\tau_A:A\lra\rc({\bf deV}(A,{\sf 2}))$$
of the dual adjunction may (as defined above) be computed as
 $$\tilde{\tau}_A(a)=\omega_A\inv(\{\mathfrak c\in{\sf Clust}(A)\st a\in \mathfrak c \})=\{\p\in{\bf deV}(A,{\sf 2})\st\p(a^*)=0\},$$
for all $a\in A$. Their complements in ${\bf deV}(A,\sf 2)$ are given by\;
${\bf deV}(A,\sf 2)\stm \tilde{\tau}_A(a)=\{\p\in {\bf deV}(A,{\sf 2})\st\p(a^*)=1\}$, and their interiors by $$\int_{{\bf deV}(A,{\sf2})}(\tilde{\tau}_A(a))=\{\p\in {\bf deV}(A,{\sf 2})\st\p(a)=1\},$$ for all $a\in A$; they form a base for open sets of ${\bf deV}(A,{\sf 2})$.
 In the hom-version, the adjunction units, for every compact Hausdorff space $X$ (see \ref{deVries}), are (by necessity) given by
\begin{center}
$\tilde{\sigma}_X = \xymatrix{[\;X\ar[rr]^{\sigma_X\qquad} & & {\sf Clust}({\sf RC}(X))\ar[rr]^{\omega_{{\sf RC}(X)}\inv\;\;} & & {\bf deV}({\sf RC}(X),{\sf 2}))\;],}$
\end{center}
so that, for all $x\in X$ and $F\in{\sf RC}(X)$, one has
$$\tilde{\sigma}_X(x)(F)=1\iff x\in {\rm int}_X(F).$$
\end{rem}

\section{Extending de Vries' duality to  Tychonoff spaces}

Having presented the de Vries duality in the form
\begin{center}
$\xymatrix{\AA^{\rm op}={\bf deV}^{\rm op}\ar@/^0.6pc/[rr]^{\quad T\,=\,{\bf deV}(-,{\sf 2})} & {\scriptstyle \simeq} & {\bf KHaus}=\XX\ar@/^0.6pc/[ll]^{\quad S\,=\,{\sf RC}}
}$
\end{center}
(see Section 6), we now wish to extend it to the category $\YY={\bf Tych}$ of all Tychonoff spaces and continuous maps, through an application of Corollary \ref{duality refl}. As ${\bf Tych}$ is reflective in ${\bf KHaus}$, to this end one considers the class $\JJ$ of the Stone-\v{C}ech-compactifications of Tychonoff spaces, that is: $\JJ$ contains precisely all continuous maps
$j:Y\to X$ with $Y\in{|\bf Tych}|$ and $ X\in |{\bf KHaus}|$, satisfying the universal property that every continuous map $f:Y\to Z$ with $Z\in |{\bf KHaus}|$ factors through $j$, by a uniquely determined continuous map $X\to Z$; equivalently, the maps $j:Y\to X$ in $\JJ$ are
%the restrictions of homeomorphisms $\beta Y\to X$ to the subspace $Y$, for a chosen Stone-\v{Cech} compactification $\beta Y$ of the Tychonoff space $Y$.
dense embeddings,  with $Y\in{|\bf Tych}|$ and $ X\in |{\bf KHaus}|$ such that $j(Y)$ is $C^*$-embedded in $X$.

According to Construction \ref{dual construction}, we now form the category
${\sf D}(\AA,\JJ,\XX)$ whose

\begin{itemize}

\item objects
are pairs $(A,j)$ with a de Vries algebra $A$ and $j:Y\to {\bf deV}(A,{\sf 2})$ in $\JJ$;

\item  morphisms $(\a,f):(A,j)\lra(A',j')$ are given by de Vries
morphisms $\a:A\lra A'$ and continuous maps $f:Y'\lra Y$ in ${\bf Tych}$
with  $T\a\circ j'=j\circ f$:
\begin{center}
$\xymatrix{TA & TA'\ar[l]_{T\a}\\
            Y\ar[u]^{j} & Y'\ar[l]^{f}\ar[u]_{j'}}$
            \end{center}

\item    composition in ${\sf D}(\AA,\JJ,\XX)$ proceeds by the horizontal pasting of diagrams;

\item the identity morphism of a ${\sf D}(\AA,\JJ,\XX)$-object $(A,j)$ is $(1_A,1_{\dom(j)})$.
\end{itemize}
The category ${\bf deV}$ is fully and coreflectively embedded into $\BB={\sf D}(\AA,\JJ,\XX)$ via
$$I:\AA\to\BB,\;(\a:A\ra A')\longmapsto (\;I\a=(\a,T\a):(A, 1_{TA})\ra(A',1_{TA'})\;).$$
From Corollary \ref{duality refl} we obtain a first extension statement for the de Vries duality:

\begin{pro}\label{deV application}
The de Vries duality extends to the duality
$$\xymatrix{{\sf D}(\AA,\JJ,\XX)^{\rm op}\ar@/^0.6pc/[rr]^{\quad\tilde{T}} & {\scriptstyle \simeq} & {{\bf Tych}}\,,\ar@/^0.6pc/[ll]^{\quad\tilde{S}}
}$$
with object assignments $\tilde{T}:(A,j:Y\to TA)\longmapsto Y$ and (see Thm. {\em \ref{hom de Vries}} for notation)
$$\tilde{S}:Y\longmapsto({\sf RC}(\beta Y),\,[\xymatrix{Y\ar@{^(->}[r] & \beta Y\ar[r]^{\tilde{\sigma}_{\beta Y}\qquad\quad} & {\bf deV}({\sf RC}(\beta Y),{\sf 2})
}]).$$
\end{pro}

\begin{rem}\label{avoiding beta}
\rm
The direct reference to the the Stone-\v{C}ech compactification in the definition of the functor $\tilde{S}$ may be avoided, as follows. Let us first note that, while the Boolean algebras $\rc(Y)$ and $\rc(\b Y)$ are isomorphic for any Tychonoff space $Y$ (see \ref{projective}), generally the contact algebras $(\rc(Y),\sfr_Y)$ and $(\rc(\b Y),\sfr_{\b Y})$ are not
as
such, unless $Y$ is a normal Hausdorff space (thanks to Urysohn's Lemma). However, under the relation $\sfr_Y^{\b}$ defined for all $F,G\in\rc(Y)$ by

\medskip

\begin{tabular}{ll}
$F\sfr_Y^\b G$ & $\iff F$  and $G$ are not completely separated\\
& $\iff \not{\exists}\;f:Y\to[0,1]$ continuous with $f(F)=0$ and $f(G)=1,$\\
\end{tabular}

\medskip

 \noindent the contact algebras  $(\rc(Y),\sfr_Y^\b)$ and $(\rc(\b Y),\sfr_{\b Y})$ become isomorphic. Consequently, the object assignment of the functor $\ti{S}$ is equivalently described by
 $$\tilde{S}:Y\longmapsto((\rc(Y),\sfr_Y^\b) ,\,[\xymatrix{Y\ar[r]^{\tilde{\sigma}_{Y}\qquad\qquad} & {\bf deV}((\rc(Y),\sfr_Y^\b) ,{\sf 2})}]).$$
 \end{rem}

 We now embark on describing the objects $(A,j:Y\to{\bf deV}(A,{\sf 2}))$ of ${\sf D}(\AA,\JJ,\XX)$ in a more algebraic fashion, without direct reference to the categories $\XX={\bf KHaus}$ or ${\bf Tych}$. First of all, replacing $Y$ by the homeomorphic subspace $j(Y)$ of ${\bf deV}(A,{\sf 2})$ one observes that $(A,j)$ is isomorphic to $(A,\,j(Y)\hookrightarrow{\bf deV}(A,{\sf 2}))$ in ${\sf D}(\AA,\JJ,\XX)$. Hence, without loss of generality, we may assume $Y$ to be a subspace of the compact Hausdorff space ${\bf deV}(A,{\sf 2})$. The task is then to express, in algebraic terms, the status of ${\bf deV}(A,{\sf 2})$ as a largest compactification of $Y$, so that (1) $Y$ is dense in ${\bf deV}(A,{\sf 2})$, and (2) any dense embedding of $Y$ into a compact Hausdorff space may be extended to a continuous map on ${\bf deV}(A,{\sf 2})$. To this end we use for all $a\in A$ the abbreviation
 $$Y_a=Y\cap\tilde{\tau}(a)=\{\p\in Y\st \p(a^*)=0\}   $$
 and employ the following terminology, which will be justified by Proposition \ref{deVries pair}.
 \begin{defi}\label{universaldeVpair}
 \rm
 (1) For a de Vries algebra $A$ and a subset $Y\sbe{\bf deV}(A,{\sf 2})$ we call $(A,Y)$ a {\em de Vries pair} if for every element $a>0$ in $A$ one has some $\p\in Y$ with $\p(a)=1$.

 (2) A de Vries pair $(A,Y)$ is called {\em universal}\/ if for any other de Vries pair $(A',Y')$ and any bijection $h:Y\to Y'$ with	$\{h\inv(Y'_{a'})\st a'\in A'\}=\{Y_a\st a\in A\}$ there is a de Vries morphism $\a:A'\to A$ with $h(\p)=\p\diamond\a$ for all $\p\in Y$.
 \end{defi}

 \begin{pro}\label{deVries pair} {\em (1)} For a de Vries algebra $A$ and a subset\/ $Y$ of\/ ${\bf deV}(A,{\sf 2})$, the pair $(A,Y)$ is a de Vries pair if, and only if, $Y$ is a dense subset of the space\/ ${\bf deV}(A,{\sf 2})$.

 {\em (2)} A de Vries pair $(A,Y)$ is universal if, and only if, the space ${\bf deV}(A,{\sf 2})$ serves as a Stone-\v{C}ech compactification of its subspace $Y$.	\end{pro}

 \begin{proof}
 (1) Since the sets $\{\p\in {\bf deV}(A,{\sf 2})\st\p(a)=1\}\,(a\in A)$ form a base for open sets in ${\bf deV}(A,{\sf 2})$, the statement is obvious.

 (2) Assuming first the universality of the de Vries pair $(A,Y)$ and considering any  compactification $c:Y\to X$ one has a homeomorphism $g:X\to{\bf deV}(A',{\sf 2})$ with a de Vries algebra $A'$. The map $g\circ c$ restricts to a homeomorphism $h:Y\to Y':=g(c(Y))$, which implies $\{h\inv(F')\st F'\in{\sf RC}(Y')\}={\sf RC}(Y)$. With \ref{projective}, this equality reads as $\{h\inv(Y'_{a'})\st a'\in A'\}=\{Y_a\st a\in A\}$. Therefore, the defining property of a universal de Vries pair gives a de Vries morphism $\a:A'\to A$ which makes $h$ a restriction of the continuous map ${\bf deV}(\a,{\sf 2}):{\bf deV}(A,{\sf 2})\to {\bf deV}(A',{\sf 2}),\,\p\mapsto\p\diamond\a$. Equivalently, the map $g\inv\circ{\bf deV}(\a,{\sf 2}):{\bf deV}(A,{\sf 2})\to X$ is a continuous extension of $c$, as desired.

 Conversely, let ${\bf deV}(A,{\sf 2})$ be a Stone-\v{C}ech compactification of its subspace $Y$ and consider a de Vries pair $(A',Y')$, along with a bijection $h:Y\to Y'$ satisfying the equality $\{h\inv(Y'_{a'})\st a'\in A'\}=\{Y_a\st a\in A\}$. This equality says precisely that $h$ and $h\inv$ map the chosen bases of closed sets in $Y$ and $Y'$ given by the regular closed sets onto each other, so that $h$ must be a homeomorphism.  Hence, enlarging its codomain to ${\bf deV}(A',{\sf 2})$, we obtain another compactification of $Y$ which, by hypotheis, must factor through a (uniquely determined) continuous map $f:{\bf deV}(A,{\sf 2})\to{\bf deV}(A',{\sf 2})$. But by the de Vries duality, we may write $f$ as $f={\bf deV}(\a,{\sf 2})$ for a (uniquely determined) de Vries morphism $\a:A'\to A$, so that $h$ becomes a restriction of ${\bf deV}(\a,{\sf 2})$, as needed.
 \end{proof}

 \begin{rem}\label{uniqueness}
 \rm The proof above shows that, in the defining property of a universal de Vries pair $(A,Y)$, the de Vries morphism $\a$ is uniquely determined by the bijection $h$.
 \end{rem}

 Definition \ref{universaldeVpair} gives us the objects of the category
  $${\bf UdeV} $$
  of universal de Vries pairs whose morphisms $\a:(A,Y)\to(A',Y')$ are de Vries morphisms $\a:A\to A'$ with $\p'\diamond\a\in Y$ for all $\p'\in Y'$;
  they get composed as in ${\bf deV} $. %is the same as in the category ${\bf deV} $.
  With
  $$f_{\a}:Y'\to Y$$
   denoting the restriction of the map ${\bf deV}(\a,{\sf 2})$, by Proposition \ref{deVries pair} one then has the ${\sf D}(\AA,\JJ,\XX)$ morphism $(\a,f_{\a}):(A,j_Y)\to(A',j_{Y'})$, with inclusion maps $j_Y, j_{Y'}$, as visualized by
\begin{center}
	$\xymatrix{{\bf deV}(A,{\sf 2}) && {\bf deV}(A',{\sf 2})\ar[ll]_{{\bf deV}(\a,{\sf 2})}\\
	Y\ar@{^(->}[u]^{j_Y} && Y'\;,\ar[ll]_{f_{\a}}\ar@{^(->}[u]_{j_{Y'}}\\
	} $	
	\end{center}
		This defines the full embedding
	$$\Psi:{\bf UdeV}\lra{\sf D}(\AA,\JJ,\XX), \;(A,Y)\longmapsto(A,j_Y).$$
	Proposition \ref{deVries pair} allows us to define also the functor
	$$\Phi: {\sf D}(\AA,\JJ,\XX)\lra{\bf UdeV}, \;(A,j:Y\to{\bf deV}(A,{\sf 2}))\longmapsto(A,j(Y)),$$
with its obvious definition on morphisms. Clearly then, $\Phi\circ\Psi={\rm Id}_{\bf UdeV}$ and $\Psi\circ\Phi\cong{\rm Id}_{	{\sf D}(\AA,\JJ,\XX)}$. Together with Proposition \ref{deV application}, this proves the following duality theorem for the category of Tychonoff spaces.	

\begin{theorem}\label{fdttychth}
The category ${\bf UdeV}$ is an equivalent retract of the category\/ ${\sf D}(\AA,\JJ,\XX)$ and therefore dually equivalent to the category $\Tych$.
\end{theorem}

The composite dual equivalence
\begin{center}
$
\xymatrix{{\bf UdeV}^{\rm op}\ar@/^0.6pc/[rr]^{\Psi\quad} & \scriptstyle{\simeq}
 &{\sf D}(\AA,\JJ,\XX)^{\rm op}\ar@/^0.6pc/[rr]^{\quad\tilde{T}}\ar@/^0.6pc/[ll]^{\Phi\quad } & {\scriptstyle \simeq} & {{\bf Tych}}\ar@/^0.6pc/[ll]^{\quad\tilde{S}}}$
\end{center}
and its natural isomorphisms $\hat{\tau},\hat{\s}$ arising from $\tilde{\tau},\tilde{\s}$ are easily described by
\begin{itemize}
\item $\tilde{T}\circ\Psi:\quad [\a:(A,Y)\to(A',Y')]\longmapsto	[f_{\a}:Y'\to Y];   $
\item $\Phi\circ\tilde{S}:\quad[f:Y'\to Y]\mapsto[ {\sf RC}(\b f):({\sf RC}(\b Y),\tilde{\s}_{\b Y}(Y))\to ({\sf RC}(\b Y'),\tilde{\s}_{\b Y'}(Y'))];$
\item $\hat{\tau}_{(A,Y)}:(A,Y)\lra\Phi\tilde{S}\tilde{T}\Psi(A,Y)=({\sf RC}({\bf deV}(A,{\sf 2})),\{\p\diamond\tilde{\tau}_A\inv\st\p\in Y\})   $, where, for simplicity, one assumes $\b Y={\bf deV}(A,{\sf 2})$;
\item $\hat{\s}_Y:Y\lra\tilde{T}\Psi\Phi\tilde{S}(Y)=\tilde{\s}_{\b Y}(Y)$.
\end{itemize}

In order to obtain the Bezhanishvili-Morandi-Olberding Duality Theorem \cite{BMO}, we now undertake a further ``algebraization'' of the objects $(A,Y)$ of ${\bf UdeV}$.
Following the key idea from \cite{BMO}, we
encode the subset inclusion $Y\hookrightarrow{\bf deV}(A,{\sf 2})$ of a de Vries pair $(A,Y)$ by a de Vries morphism $A\to{\sf P}(Y)$, with the complete atomic Boolean algebra ${\sf P}(Y)$ regarded as a discrete de Vries algebra (so that $\ll_{\circ}\;=\;\subseteq$; see \ref{deVriesAlgebraic}). As $Y$ is recovered from ${\sf P}Y$ as its set of atoms, no loss of information occurs in the process, as we make precise in what follows.

\begin{defi}\label{associated deV ext}
\rm  	
Given a de Vries pair $(A,Y)$, maintaining the notation of Theorem \ref{hom de Vries},
we call the composite map
\begin{center}
$\gamma_{(A,Y)} = \xymatrix{[\;A\ar[rr]^{\tilde{\tau}_A\qquad\quad} & & {\sf RC}({\bf deV}(A,{\sf 2}))\ar[r]^{\rm int}_{\cong} &  {\sf RO}({\bf deV}(A,{\sf 2}))\ar[rr]^{\quad Y\cap(-)}& &{\sf P}(Y)\;]\; }$
\end{center}
the {\em Booleanization} of $(A,Y)$. By definition, $\gamma_{(A,Y)}$ is for every $a\in A$ described by
$$\gamma_{(A,Y)}(a)=Y\cap{\rm int}(\tilde{\tau}(a))=Y\cap\{\p\in{\bf deV}(A,{\sf 2})\st \p(a)=1\}= \{\p\in Y\st\p(a)=1\}.$$
\end{defi}
A proposition similar to the next one was established in \cite{BMO}, but in another setting.

\begin{pro}\label{UBMO prop} Let $(A,Y)$ be a de Vries pair.

\medskip
{\em (1)}
	With the complete atomic Boolean algebra ${\sf P}(Y)$
regarded as a discrete de Vries algebra,
 the Booleanization $\gamma_{(A,Y)}:A\to {\sf P}(Y)$ of $(A,Y)$
	is an injective de Vries morphism, and every atom in ${\sf P}(Y)$ may be presented as a meet of elements in the image of $\gamma_{(A,Y)}$.

\medskip
	
	{\em(2)}
	For any other de Vries pair $(A',Y')$ and any bijection $h:Y\to Y'$  satisfying (in the notation of Definition {\em \ref{deVries pair}})
$\{h\inv(Y'_{a'})\st a'\in A'\}=\{Y_a\st a\in A\}$,
\begin{itemize}
\item[\em(a)]	
the de Vries morphism\/ ${\sf P}(h)\diamond \gamma_{(A',Y')}$ has the same image in\/ ${\sf P}(Y)$ as\/  $\gamma_{(A,Y)}$, and
	\item[\em (b)]
if the de Vries pair $(A,Y)$ is universal, then
  there exists a unique de Vries morphism $\alpha:A'\to A$ %with $\gamma\diamond \alpha=\gamma'$.
  rendering the diagram
 \begin{center}$\xymatrix{A'\ar[r]^{\a}\ar[d]_{\gamma_{(A',Y')}} & A\ar[d]^{\gamma_{(A,Y)}}\\
 {\sf P}(Y')\ar[r]^{{\sf P}(h)} & {\sf P}(Y)\\
 }$	
 \end{center}
  commutative in the category ${\bf deV}$, that is: ${\sf P}(h)\diamond\gamma_{(A',Y')}=\gamma_{(A,Y)}\diamond\a$.
	\end{itemize}
\end{pro}
\begin{proof}
(1)  To see that $\gamma:=\gamma_{(A,Y)}$ is injective, with $X:={\bf deV}(A,{\sf 2})$ it suffices to show that the map
$Y\cap (-):\ro(X)\lra {\sf P}(Y)$ is injective. But this map is just a codomain enlargement
of $Y\cap (-):\ro(X)\lra \ro(Y)$, which is the composite of the Boolean isomorphisms $\cl_X:\ro(X)\lra\rc(X)$,  $r=Y\cap(-):\rc(X)\lra\rc(Y)$ (see \ref{projective}), and $\int_Y:\rc(Y)\lra\ro(Y)$.

Furthermore, since the sets ${\rm int}(\tilde{\tau}(a))\;(a\in A)$ form a base for open sets in $X$, the sets $\g(a)\;(a\in A)$ form a base for open sets in $Y$. Every singleton set $\{\p\}$ in this
$T_1$-space, {\em i.e.} every atom of ${\sf P}(Y)$, is an intersection of its basic open neighbourhoods in $Y$, and these lie all in the image of $\g$.

We verify that $\g$ is a de Vries morphism. Indeed, as for every de Vries morphism $\p:A\to{\sf 2}$ one has  $\p(0)=0$, trivially  $\gamma(0)=\emptyset$ follows, {\em i.e.}, $\g$ satisfies (V1). Similarly, one derives the satisfaction of (V2-4) for $\g$ from the corresponding properties of the de Vries morphisms $\p$, as follows. (V2): Since $1=\p(a\we b)=\p(a)\we\p(b)$ means equivalently $\p(a)=1=\p(b)$ for  all $\p$, the equality $\g(a\we b)=\g(a)\cap\g(b)$ follows.\\
(V3):\ Since $a^*\ll b$ in $A$ implies $\p(a)^*\ll\p(b)$ for all $\p$, so that when $\p(a)=0$ one must have $\p(a)^*=1$ and then also $\p(b)=1$, we conclude $\g(a)^*=\{\p\in Y\st \p(a)=0\}\subseteq\g(b)$. (V4): Since $\p(a)=\bigvee\{\p(b)\st b\ll a\}$, for all $\p$ one has $\p(a)=1$ precisely when $\p(b)=1$ for some $b\ll a$ in $A$. Hence, $\g(a)=\bigcup\{\g(b)\st b\ll a\}$.

 \medskip

 (2)(a) Let us note first that $\PPP(h):\PPP(Y\ap)\to\PPP(Y)$ is a (sup-preserving) Boolean isomorphism and, hence, a de Vries morphism satisfying $\PPP(h)\diamond \g'={\sf P}(h)\circ\g'$, where $\g'=\g_{(A',Y')}$.
 Since $\g(A)=\ro(Y)$ (see the proof of (1)) and, likewise, $\g'(A\ap)=\ro(Y\ap)$,
 and since the hypotheses make $h:Y\to Y\ap$ a homeomorphism of the respective subspaces of ${\sf deV}(A,{\sf 2})$ and ${\sf deV}(A',{\sf 2})$, we conclude
 $$({\sf P}(h)\diamond\g')(A')={\sf P}(h)({\sf RO}(Y'))=\ro(Y)=\g(A).$$

  (b)
   The universality of the de Vries pair  $(A,Y)$ gives us a uniquely determined de Vries morphism
 $\a: A'\to A$ with $h(\p)=\p\di\a$ for all $\p\in Y$. We claim that this last equality implies the commutativity of the diagram above, {\em i.e.}, that $h\inv(\g'(a'))=(\g\di\a)(a')$ holds for all $a'\in A'$, and that this implication is in fact reversible.

 Indeed, for the necessity of the condition we have that, for all $a'\in A'$ and $\p\in Y$,

 \medskip

 \begin{tabular}{ll}
$\p\in h\inv(\g\ap(a'))$&$\iff h(\p)\in\g(a')$ \\
&$\iff   h(\p)(a')=1$\\
& $\iff(\p\di\a)(a')=1$ \qquad\qquad\qquad\qquad  (by hypothesis)\\
&$\iff\exists\; b'\ll a'	\text{ in }A': \p(\a(b'))=1$\\
&$\iff\exists\; b'\ll a'	\text{ in }A': \p\in\g(\a(b'))$\\
&$\iff\p\in (\g\di\a)(a')$.\\
\end{tabular}

\medskip

\noindent Conversely, assuming the commutativity of the diagram and reshuffling the above equivalences, one obtains for all $\p\in Y$ and $a'\in A'$ the equivalence
$$h(\p)(a')=1\iff(\p\di\a)(a')=1,$$
which means  $h(\p)=\p\di\a$ for all $\p\in Y$. In conclusion, while the necessity of the condition shows the existence of the desired de Vries morphism $\a$, the sufficiency implies the uniqueness of $\a$, by Remark \ref{uniqueness}.
\end{proof}

\begin{rem}
\rm
The de Vries duality associates the de Vries algebra ${\sf 2}$ with the singleton space $\sf 1\cong{\bf deV}({\sf 2,\sf 2})$ which, in categorical terms, is a regular generator (also called separator) in the category ${\bf KHaus}$ (since every compact Hausdorff space is a quotient of the Stone-\v{C}ech compactification of its underlying set, considered as a discrete space).
Consequently, $\sf 2$ is a regular cogenerator (or coseparator) in the category ${\bf deV}$. This makes every de Vries algebra a subalgebra of some power of $\sf 2$. The injectivity of the map $\g_{(A,Y)}$ of Proposition \ref{UBMO prop} may be seen as a consequence of this important categorical role of the de Vries algebra $\sf 2$.

Moreover, $\sf 2$ is regular injective in $\bf deV$ since, trivially, the singleton space $\sf  1$ is regular projective in $\bf KHaus$, {\em i.e.}, projective with respect to regular epimorphisms in ${\bf KHaus}$, which are the surjective continuous maps.
\end{rem}

Proposition \ref{UBMO prop} motivates the following definition	

\begin{defi}\label{UBMO def0}
\rm
(1) An injective de Vries morphism $\gamma:A\to B$ into a complete atomic Boolean algebra $B$ (regarded as a discrete de Vries algebra) is called a {\em Boolean de Vries extension} (or just a {\em de Vries extension}, \cite{BMO}) of the de Vries algebra $A$ if every atom in $B$ is a meet of elements in the image of $\g$.

(2) A Boolean de Vries extension $\gamma:A\to B$ is called {\em universal} (or {\em maximal}, \cite{BMO}), if every Boolean de Vries extension $\g\ap:A\ap\to B$ into the same Boolean algebra $B$ with $\g(A)=\g\ap(A\ap)$ factors as $\g'=\g\di\a$, for some de Vries morphism $\a: A'\to A$.
\end{defi}

We should clarify immediately how to regard a Boolean de Vries extension as a de Vries pair.
Considering the fact that the complete Boolean homomorphisms $\xi:B\to{\sf 2}$ of the complete atomic Boolean algebra $B$ correspond bijectively to the atoms $x$ in $B$, so that the map $$\vk_B:{\sf At}(B)\lra {\bf CABA}(B,{\sf 2}),\ \ x\mapsto \xi_x,$$ where $(\xi_x(b)=1\iff x\leq b)$ for all $b\in B$, is a natural bijection
 (see \ref{Tarski}), we define:

\begin{defi}\label{Theta}
\rm
For a Boolean de Vries extension $\g:A\to B$, we consider the hom map
$${\bf deV}(\g,{\sf 2}):{\bf deV}(B,{\sf 2})\lra{\bf deV}(A,{\sf 2}),\quad\xi\longmapsto \xi\diamond\g,$$
and then the image $Y^{\g}$ of ${\bf CABA}(B,2)\subseteq{\bf deV}(B,2)$ under this map:
$$Y^{\g}=\{\xi\diamond\g\st\xi\in{\bf CABA}(B,{\sf 2})\}=\{\xi\circ\g\st\xi\in{\bf CABA}(B,{\sf 2})\}.$$
Since, for every $a>0$ in $A$, the injectivity of $\g$ guarantees $\g(a)>0$ and, thus,
$\g(a)\geq x$ for some $x\in{\sf At}(B)$, we obtain $(\xi_x\di\g)(a)=\xi_x(\g(a))=1$ with $\xi_x\di\g\in Y^{\g}$. Hence, $(A,Y^{\g})$ is a de Vries pair, and we call it {\em induced}\/ by $\g$ and $Y^{\g}$ the {\em trace} of $\g$.
\end{defi}

With $\g_{\rm at}$ denoting the (``atomic'') restriction of ${\bf deV}(\g,{\sf 2})$
we have the trivially commuting diagram
\begin{center}
$\xymatrix{{\bf deV}(A,{\sf 2}) & {\bf deV}(B,{\sf 2})\ar[l]_{{\bf deV}(\g,{\sf 2})}\\
Y^{\g}\ar@{^(->}[u] & {\bf CABA}(B,{\sf 2})\,.\ar@{^(->}[u]\ar[l]_{\g_{\rm at}\quad }\\
}$	
\end{center}
We regard $Y^{\g}$ as a subspace of the compact Hausdorff space ${\bf deV}(A,{\sf 2})$
and, likewise, ${\bf CABA}(B,{\sf 2})$ as a subspace of the Stone space ${\bf deV}(B,{\sf 2})\cong {\sf Ult}(B)$ and first prove:

\begin{pro}\label{atomic lemma}
 {\em (1)} For every Boolean de Vries extension $\g:A\lra B$, the map $\g_{\rm at}:{\bf CABA}(B,{\sf 2})\lra Y^{\g}\;,\quad\xi\longmapsto\xi\circ\g,$
 is a
 continuous bijection.

 {\em (2)} The Booleanization $\g_{(A,Y)}$ of a universal de Vries pair $(A,Y)$ is a universal Boolean de Vries extension.
\end{pro}
\begin{proof}
(1) As $\g_{\rm at}$ is a surjective subspace restriction of the continuous map ${\bf deV}(\g,{\sf 2})$, it suffices to prove that $\g_{\rm at}$ is injective.	 Every atom $x$ in $B$ may be written as $x=\bigwedge_{i\in I}\g(a_i)$. Hence, with the complete Boolean homomorphism $\xi_x:B\to{\sf 2}$ representing $x$, one has $1=\xi_x(x)=\bw_{i\in I}\xi(\g(a_i))$. Assuming $\xi_x\circ\g=\xi_z\circ\g$ with $\xi_z$ representing $z\in{\sf At}(B)$, we then obtain $\xi_z(x)=\bw_{i\in I}\xi_z(\g(a_i))=\bw_{i\in I}\xi_x(\g(a_i))=1$, or $z\leq x$.
Hence $z=x$ and thus $\xi_x=\xi_z$.

(2) By Proposition \ref{UBMO prop}(1), $\g=\g_{(A,Y)}:A\to{\sf P}(Y)=B$  is a Boolean de Vries extension. To confirm its universality, let $\g':A'\to B$ be an injective de Vries morphism with $\g'(A')=\g(A)$. With (1) we have that the composite map
\begin{center}
$\xymatrix{
h:=(Y\ar[r]^{\vk_B\quad} & {\bf CABA}(B,{\sf 2})\ar[r]^{\quad \g'_{\rm at}} & Y^{\g'}=:Y'),\\
}$
$\p\longmapsto \xi_{\p}\diamond\g',$
\end{center} is a bijection.	 Moreover, the Boolean isomorphism ${\sf P}(h)$ makes the Boolean de Vries extension $\g':A'\to B$ factor through the Booleanization of the de Vries pair $(A',Y')$, induced by $\g'$, as in
\begin{center}
$\xymatrix{& {\sf P}(Y')\ar[rd]^{{\sf P}(h)} & \\
A'\ar[ur]^{\g_{(A',Y')}}\ar[rr]^{\g'} & & {\sf P}(Y)\,.\\
}$	
\end{center}

\noindent Indeed, for all $a'\in A'$ and $\p\in Y$ one has
\smallskip

\begin{tabular}{ll}
$\p\in({\sf P}(h)\di\g_{(A',Y')})(a')$ & $\iff h(\p)\in\g_{(A',Y')}(a')$\\
&$\iff h(\p)(a')=1$\\
&$\iff \xi_{\p}(\g'(a'))=1$\\
&$\iff \p\in\g\ap(a\ap)$.\\
\end{tabular}

\medskip
\noindent Consequently,
$${\sf P}(h)({\sf RO}(Y'))={\sf P}(h)(\g_{(A',Y')}(A'))=\g'(A')=\g(A)={\sf RO}(Y);$$ equivalently, ${\sf P}(h)({\sf RC}(Y'))={\sf RC}(Y),$ or, in the notation of Proposition
\ref{UBMO prop}(2),
$$\{h\inv(Y'_{a'})\st a'\in A'\}=\{Y_a\st a\in A\},$$ which then gives the existence of a
(unique) de Vries morphism $\alpha:A'\to A$ with $\gamma\diamond \alpha=\gamma'$.
\end{proof}

The universal Boolean de Vries extensions are the objects of the category
$${\bf UBdeV}$$
whose morphisms $(\alpha,\delta):\gamma\to\gamma'$ are given by commutative squares
\begin{center}
$\xymatrix{A\ar[r]^{\a}\ar[d]_{\g} & A'\ar[d]^{\g'}\\
B\ar[r]^{\delta} & B'\\
}$	
\end{center}
in ${\bf deV}$, with a de Vries morphism $\a$ and a complete Boolean homomorphism $\delta$, so that  $\delta\diamond\g=\g'\diamond\a $;
note that $\delta\diamond \g$ is simply the map composite $\delta\circ\g$ since $\delta$ preserves suprema. The composition in ${\bf UBdeV}$ proceeds as in the arrow category of ${\bf deV}$, {\em i.e.}, by horizontal pasting of diagrams, so that ${\bf UBdeV}$ is in fact a full subategory of ${\bf deV}^{\sf 2}$.

\begin{pro}\label{Gamma functor}
{\em (1)} Assigning to every universal de Vries pair $(A,Y)$ its Booleani-\ zation $\g_{(A,Y)}:A\to{\sf P}(Y)$	 (as in Definition {\em \ref{associated deV ext}}) describes the object map of a functor
$$\Gamma:{\bf UdeV} \lra{\bf UBdeV}\,.$$
{\em (2)} Assigning to every universal Boolean de Vries extension $\g:A\to B$ its induced de Vries pair $(A,Y^{\g})$ (as in Definition {\em \ref{Theta})} describes the object map of a functor
$$\Delta: {\bf UBdeV}\lra{\bf UdeV}\,.$$
\end{pro}
\begin{proof}
(1) By Proposition \ref{atomic lemma}(1), $\Gamma$ is well-defined on objects.
 Further,  for a morphism $\a:(A,Y)\to(A',Y')$ in ${\bf UdeV} $, one lets $\GA(\a) =(\a,\PPP(f_{\a})):\g_{(A,Y)}\to\g_{(A',Y')}$ and must confirm the commutativity of the diagram %Then $\Gamma(\a,f):\Gamma(A,Y)\to\Gamma(A',Y')$ is given by the diagram
\begin{center}
$\xymatrix{A\ar[rr]^{\a}\ar[d]_{\g_{(A,Y)}} && A'\ar[d]^{\g_{(A',Y')}}\\
{\sf P}(Y)\ar[rr]^{\PPP(f_{\a})} && {\sf P}(Y')\\
}.$	
\end{center}
But this may be be seen similarly as at the end of
 the proof of
 Proposition \ref{UBMO prop}(2). Indeed, for all $a\in A$ and $\p'\in Y'$ one has
\begin{align*}
\p'\in \PPP(f_{\a})(\g_{(A,Y)}(a)) &\iff f_{\a}(\p\ap)\in \g_{(A,Y)}(a)\\
&\iff (\p'\diamond\a)(a)=1\\
&\iff \exists\, b\ll a:\p'(\a(b))=1\\
&\iff \p'\in\bigcup\{\g_{(A',Y')}(\a(b))\st b\ll a\}\\
& \iff \p'\in (\g_{(A',Y')}\diamond\a)(a) \,.\\
\end{align*}
The functoriality of  $\Gamma$ is now trivial.

We must verify that the de Vries pair $(A,Y^{\g})$ is universal when the Boolean de Vries extension $\g:A\to B$ is universal.	To this end, using the notation of \ref{Tarski} and Definitions \ref{associated deV ext} and \ref{Theta}, one easily confirms that the diagram
\begin{center}
$\xymatrix{ A\ar[r]^{\g}\ar[d]_{\g_{(A,Y^{\g})}} & B\ar[d]^{\tilde{\theta}_B}\\
{\sf P}(Y^\g)\ar[r]^{{\sf P}(\g_{\rm at})\qquad} & {\sf P}({\bf CABA}(B,{\sf 2}))\\
}$	
\end{center}
commutes. Hence, with the Boolean isomorphism $\nu_B:=\tilde{\theta}_B\inv\circ{\sf P}(\g_{\rm at})$ and with $Y:=Y^{\g}$, the map $\nu_B\circ\g_{(A,Y)}$ has the same image as $\g$ in $B$.
Now, given any de Vries pair $(A',Y')$ and a bijection $h:Y\to Y'$ as in Definition \ref{universaldeVpair}(2), from Proposition \ref{UBMO prop}(2)(a) we know that ${\sf P}(h)\circ\g_{(A',Y')}$ has the same image as $\g_{(A,Y)}$ in ${\sf P}(Y)$. Consequently, the Boolean de Vries extension $\nu_B\circ{\sf P}(h)\circ\g_{(A',Y')}$ has the same image in $B$ as $\g$ and must therefore factor through $\g$ in ${\bf deV}$, as shown in the diagram
\begin{center}
$\xymatrix{A'\ar@{-->}[rr]^{\a}\ar[d]_{\g_{(A',Y')}} && A\ar[d]^{\g}\ar[ld]_{\g_{(A,Y)}}\\
{\sf P}(Y')\ar[r]_{{\sf P}(h)} & {\sf P}(Y)\ar[r]_{\nu_B} &B\,.\\
}$	
\end{center}
As at the end of the proof of Proposition \ref{UBMO prop}(2)(b), from  ${\sf P}(h)\di\g_{(A',Y')}=\g_{(A,Y)}\di\a$ one concludes that $h(\p)=\p\di\a$ holds for every $\p\in Y$. This confirms the universality of the de Vries pair $(A,Y^{\g})$.
	
	For a morphism $(\a,\delta):(\g:A\to B)\lra(\g':A'\to B')$ in ${\bf UBdeV}$,  %as in Definition \ref{UBMO def}
	since
	$$(\xi'\circ\gamma')\diamond\a=\xi'\circ(\g'\diamond\a)=\xi'\circ(\delta\circ\g)=(\xi'\circ\delta)\circ\g$$
	for all $\xi'\in{\bf CABA}(B',{\sf 2})$, the hom map
	${\bf deV}(\a,{\sf 2}):{\bf deV}(A',{\sf 2})\lra{\bf deV}(A,{\sf 2})$ maps $Y^{\g'}$ into $Y^{\g}$. Hence, $\Delta(\a,\delta):=\a:(A,Y^{\g})\to(A',Y^{\g'})$ is a morphism in ${\sf UdeV}$, and the functoriality of $\Delta$ is trivial.
	\end{proof}

We are now ready to prove the main theorem of this section.

\begin{theorem}\label{BMO theorem}
The functors $\Gamma$ and $\Delta$ %: \DVT\longrightarrow {\bf UBMO}$
of\/ {\em Proposition \ref{Gamma functor}} satisfy
$$\Delta\circ\Gamma={\rm Id}_{\bf UdeV} \quad\text{ and }\quad \Gamma\circ\Delta\cong{\rm Id}_{\bf UBdeV}$$
and therefore exhibit the category $\bf UdeV$ as an equivalent retract of the category $\bf UBdeV$.
\end{theorem}
\begin{proof}
For every universal de Vries pair
$(A,Y)$, one has
$\Delta(\Gamma(A,Y))=\Delta(\g_{(A,Y)})=(A,Y^{\g_{(A,Y)}})$, where $$Y^{\g_{(A,Y)}}\!=\!\{\xi\circ\g_{(A,Y)}\st \xi\in{\bf CABA}({\sf P}(Y),{\sf 2})\}\hookrightarrow{\bf deV}(A,{\sf 2}).$$
Hence, it suffices to show $Y^{\g_{(A,Y)}}=Y$ to be able to deduce $\Delta(\Gamma(A,Y))=(A,Y)$. But under the natural bijection $\chi_Y:Y\to{\bf CABA}({\sf P}(Y),{\sf 2})$ of \ref{Tarski}, we can write $\xi\in{\bf CABA}({\sf P}(Y),{\sf 2})$ equivalently as $\chi_Y^{\p}$ with $\p\in Y$, and for all $a\in A$ one has
$$ (\chi_Y^{\p}\circ\g_{(A,Y)})(a)=1\iff\chi_Y^{\p}(\{\psi\in Y\st \psi(a)=1\})=1\iff\p(a)=1\,,$$
so that $\chi_Y^{\p}\circ\g_{(A,Y)}=\p$. Hence, $Y^{\g_{(A,Y)}}=Y$ follows. That $\Delta\circ\Gamma$ maps morphisms identically as well is now obvious.

For every universal Boolean de Vries extension $\g:A\to B$, we compute
$$\Gamma(\Delta(\g))=\Gamma(A,\,Y^{\g})\!=\!\{\xi\circ\g\st\xi\in{\bf CABA}(B,{\sf 2})\}\,)=(\g_{(A,Y^{\g})}:A\longrightarrow {\sf P}(Y^{\g})),$$
with $\g_{(A,Y^{\g})}(a)=\{\xi\circ\g\st \xi\in {\bf CABA}(B,{\sf 2}),\, \xi(\g(a))=1\}$, for all $a\in A$. But with the Tarski isomorphism $\tilde{\theta}_B$ (see \ref{Tarski}) and the bijection $\g_{\rm at}$ of Proposition \ref{atomic lemma}(1) we have the natural composite Boolean isomorphism
\begin{center}
$\xymatrix{B\ar[r]^{\tilde{\theta}_B\qquad\quad} & {\sf P}({\bf CABA}(B,{\sf 2}))\ar[r]^{\qquad{\sf P}(\g_{\rm at}\inv)} & {\sf P}(Y^{\g}),\\
}$\hfil $b\mapsto\{\xi\st\xi(b)=1\}\mapsto\{\xi\circ\g\st	\xi(b)=1\}.$
\end{center}
So, for every $a\in A$, this isomorphism maps $\g(a)$ precisely to the set $\g_{(A,Y^{\g})}(a)$, and we obtain the commutative diagram
\begin{center}
$\xymatrix{A\ar[rr]^{1_A}\ar[d]_{\g} && A\ar[d]^{\g_{(A,Y^{\g})}}\\
B\ar[rr]^{{\sf P}(\g_{\rm at}\inv)\circ\ti{\theta}_B\;} && {\sf P}(Y^{\g})
}$
\end{center}
which represents an isomorphism $\g\to \g_{(A,Y^{\g})}$ in ${\bf UBdeV}$. The confirmation of its naturality with respect to $\g$ involves only a routine check.	\end{proof}

\begin{cor}
The Bezhanishvili-Morandi-Olberding duality	{\em \cite{BMO}} may be obtained as the composite
of the equivalences of Proposition {\em \ref{deV application}} and Theorems {\em \ref{fdttychth}} and {\em \ref{BMO theorem}}:
\begin{center}
$
\xymatrix{{\bf UBdeV}^{\rm op}\ar@/^0.6pc/[rr]^{\quad\Delta\quad} & \scriptstyle{\simeq} &{\bf UdeV} ^{\rm op}\ar@/^0.6pc/[ll]^{\quad\Gamma\quad}\ar@/^0.6pc/[rr]^{\Psi\;} & \scriptstyle{\simeq}
 &{\sf D}(\AA,\JJ,\XX)^{\rm op}\ar@/^0.6pc/[rr]^{\quad\tilde{T}}\ar@/^0.6pc/[ll]^{\Phi\;} & {\scriptstyle \simeq} & {{\bf Tych}\;.}\ar@/^0.6pc/[ll]^{\quad\tilde{S}}}$
\end{center}
\end{cor}

Here the composite dual equivalence $\tilde{T}\circ\Psi\circ\Delta$ assigns to a universal Boolean de Vries extension $\g:A\to B$ its trace $Y^{\g}=\{\xi\circ\g\st\xi\in{\bf CABA}(B,{\sf 2})\}$ in ${\bf deV}(A,{\sf 2})$, and $\Gamma\circ\Phi\circ\tilde{S}$ first embeds a Tychnoff space $Y$ into the compact space ${\bf deV}(A,{\sf 2})$ with $A={\sf RC}(\b Y)$ (or, more simply, $A={\sf RC}(Y)$, at the expense of having to use a less simple contact relation for the de Vries structure--see Remark \ref{avoiding beta}) and then sends it to the Booleanization $\g_{(A,Y)}:A\to{\sf P}(Y)$ of the universal de Vries pair $(A,Y)$, where we have identified $Y$ with its image under the embedding into ${\bf deV}(A,{\sf 2})$.

\end{document}